\documentclass[11pt,a4paper]{amsart}

\usepackage{amssymb,mathrsfs,enumitem,amsmath,amsthm,bbold}
\usepackage[mathscr]{euscript}
\usepackage{color}

\usepackage[all]{xy}
\definecolor{Chocolat}{rgb}{0.36, 0.2, 0.09}
\definecolor{BleuTresFonce}{rgb}{0.215, 0.215, 0.36}
\definecolor{EgyptianBlue}{rgb}{0.06, 0.2, 0.65}

\usepackage[OT2,OT1]{fontenc}

\newcommand\cyrillic[1]{
	{\fontencoding{OT2}\fontfamily{wncyr}\selectfont #1}
		}

\newcommand\mathcyr[1]{\text{\cyrillic{#1}}}
\newcommand\Sha{\textnormal{\mathcyr{Sh}}} 

\usepackage[colorlinks,final,hyperindex]{hyperref}
\hypersetup{citecolor=BleuTresFonce, urlcolor=EgyptianBlue, linkcolor=Chocolat}

\usepackage{fourier}

\newtheorem*{itheorem}{Theorem}
\newtheorem{theorem}{Theorem}[section]
\newtheorem{corollary}[theorem]{Corollary}

\newtheorem{lemma}[theorem]{Lemma}
\newtheorem{proposition}[theorem]{Proposition}

\newtheorem{conjecture}[theorem]{Conjecture}

\theoremstyle{definition}

\newtheorem{remark}[theorem]{Remark}

\DeclareMathAlphabet{\pazocal}{OMS}{zplm}{m}{n}

\def\calL{\pazocal{L}}

\def\calN{\pazocal{N}}
\def\calO{\pazocal{O}}
\def\calP{\pazocal{P}}

\def\calR{\pazocal{R}}
\def\calS{\pazocal{S}}
\def\calT{\pazocal{T}}

\def\calX{\pazocal{X}}
\def\calY{\pazocal{Y}}
\def\calZ{\pazocal{Z}}

\DeclareMathOperator{\Lie}{Lie}

\DeclareMathOperator{\Ass}{Ass}

\DeclareMathAlphabet{\mathbbold}{U}{bbold}{m}{n}

\def\k{\mathbbold{k}}

\newcommand{\twocor}[3]{\ensuremath{ 
\vcenter{\hbox{\xymatrix@R=.4pc@C=.2pc{ %@R=.7mm@C=1mm{
    #2\ar@{-}[dr] & &#3\ar@{-}[dl]\\
	&*+[o][F-]{#1}\ar@{-}[d]&\\
	&*{}&
}}}}}

\newcommand{\lbincomb}[5]{\ensuremath{
\vcenter{\hbox{\xymatrix@R=.4pc@C=.2pc{ %}@R=.7mm@C=1mm{
			#3\ar@{-}[dr] &&#4\ar@{-}[dl] &\\
			&*+[o][F-]{#2}\ar@{-}[dr]&&#5\ar@{-}[dl]\\
			&&*+[o][F-]{#1}\ar@{-}[d]&\\
			&&*{}&
}}}}}
	
\newcommand{\rbincomb}[5]{\ensuremath{
\vcenter{\hbox{\xymatrix@R=.4pc@C=.2pc{ %}@R=.7mm@C=1mm{
			 &#4\ar@{-}[dr] &&#5\ar@{-}[dl] \\
			#3\ar@{-}[dr]&&*+[o][F-]{#2}\ar@{-}[dl]\\
			&*+[o][F-]{#1}\ar@{-}[d]&&\\
			&*{}&&
}}}}}

\newcommand{\binfork}[7]{\!\!\!\ensuremath{
\vcenter{\hbox{\xymatrix@R=.4pc@C=.2pc{ %@R=.7mm@C=1mm{
			#4\ar@{-}[dr] && #5\ar@{-}[dl]\,\,#6\ar@{-}[dr]&&#7\ar@{-}[dl]\\
			&*+[o][F-]{#2}\ar@{-}[dr] &   &*+[o][F-]{#3}\ar@{-}[dl]\\
			&&*+[o][F-]{#1}\ar@{-}[d]&&\\
			&&*{}&&
		}}}}\!\!\!\!}

\newcommand{\llbincomb}[7]{\ensuremath{\!\!\!\!
 \vcenter{\hbox{\xymatrix@R=.4pc@C=.2pc{ %@R=.7mm@C=1mm{
 			#4\ar@{-}[dr] && #5\ar@{-}[dl]&&\\
 			&*+[o][F-]{#3}\ar@{-}[dr]&&#6\ar@{-}[dl] \\
 			&&*+[o][F-]{#2}\ar@{-}[dr] &   &#7\ar@{-}[dl]\\
 			&&&*+[o][F-]{#1}\ar@{-}[d]&\\
 			&&&*{}&
 		}}}}}

\newcommand{\lbtcomb}[6]{\ensuremath{
\vcenter{\hbox{\xymatrix@R=.4pc@C=.2pc{ %}@R=.7mm@C=1mm{
			#3\ar@{-}[dr] &#4\ar@{-}[d]& #5\ar@{-}[dl]&\\
			&*+[o][F-]{#2}\ar@{-}[dr]&&#6\ar@{-}[dl]\\
			&&*+[o][F-]{#1}\ar@{-}[d]&\\
			&&*{}&
}}}}}

\newcommand{\ltbcomb}[6]{\ensuremath{
\vcenter{\hbox{\xymatrix@R=.4pc@C=.2pc{ %}@R=.7mm@C=1mm{
			#3\ar@{-}[dr] &&#4\ar@{-}[dl] &\\
			&*+[o][F-]{#2}\ar@{-}[dr]&#5\ar@{-}[d]&#6\ar@{-}[dl]\\
			&&*+[o][F-]{#1}\ar@{-}[d]&\\
			&&*{}&
}}}}}

\newcommand{\rrbincombfork}[9]{\ensuremath{
\vcenter{\hbox{\xymatrix@R=.4pc@C=.2pc{ %@R=.7mm@C=1mm{
			&#6\ar@{-}[dr] && #7\ar@{-}[dl]&#8\ar@{-}[dr] && #9\ar@{-}[dl]\\
			&&*+[o][F-]{#5}\ar@{-}[dr]& && *+[o][F-]{#3}\ar@{-}[dll]&\\
			&#4\ar@{-}[dr] &   &*+[o][F-]{#2}\ar@{-}[dl]&&\\
			&&*+[o][F-]{#1}\ar@{-}[d]&&&\\
			&&*{}&&&&
}}}}\!\!\!\!\!\!\!\!\!\!}

\begin{document}

\title{An effective criterion for Nielsen--Schreier varieties}

\author{Vladimir Dotsenko}

\address{ 
Institut de Recherche Math\'ematique Avanc\'ee, UMR 7501, Universit\'e de Strasbourg et CNRS, 7 rue Ren\'e-Descartes, 67000 Strasbourg CEDEX, France}

\email{vdotsenko@unistra.fr}

\author{Ualbai Umirbaev}

\address{Department of Mathematics, Wayne State University, Detroit, MI 48202, USA; Department of Mathematics, Al-Farabi Kazakh National University, Almaty, 050040, Kazakhstan; Institute of Mathematics and Mathematical Modeling, Almaty, 050010, Kazakhstan,}

\email{umirbaev@wayne.edu}

\dedicatory{To the memory of V.~A.~Artamonov (1946--2021)}

\date{}

\begin{abstract}
All algebras of a certain type are said to form a Nielsen--Schreier variety if every subalgebra of a free algebra is free.  This property has been perceived as extremely rare; in particular, only six Nielsen--Schreier varieties of algebras with one binary operation have been discovered in prior work on this topic. We propose an effective combinatorial criterion for the Nielsen--Schreier property in the case of algebras over a field of zero characteristic; in our approach, operads play a crucial role. Using this criterion, we show that the well known varieties of all pre-Lie algebras and of all Lie-admissible algebras are Nielsen--Schreier, and, quite surprisingly, that there are already infinitely many non-equivalent Nielsen--Schreier varieties of algebras with one binary operation and identities of degree four. 
\end{abstract}

\maketitle

%\setcounter{tocdepth}{2}
%\tableofcontents

\section{Introduction}

A classical theorem of combinatorial group theory asserts that every subgroup of a free group is free. Nielsen established this remarkable result in 1924 for finitely generated subgroups \cite{MR1512188};  in 1927, it was extended to all subgroups by Schreier \cite{MR3069472}. More generally, algebras satisfying certain identities are said to form a Nielsen--Schreier variety of algebras if every subalgebra of a free algebra is free. Five Nielsen--Schreier varieties of algebras with one binary operation have been known from classical results in the ring theory going back to the 1950s: all algebras (Kurosh \cite{MR0020986}), all Lie algebras (Shirshov and Witt \cite{MR0059892,MR77525}), all commutative or anticommutative algebras (Shirshov \cite{MR0062112}), and the trivial example of all algebras with zero product. 

The problem of classifying all Nielsen--Schreier varieties of algebras was originally recorded in the 1976 edition of Dniester Notebook by V. A. Parfenov (see the easily accessible English translation of a later edition \cite[Question 1.179]{MR2203726}), and then reiterated in \cite[Problem 1.1]{MR2021795} and in \cite[Problem 11.3.9]{MR2014326}. Additionally, the same question is raised in the survey on ``niceness theorems'' by Hazewinkel~\cite{MR2562563}, who writes (about the existing general freeness result of Fresse~\cite{MR1639900}):
\begin{quote}
I don't think it can be made to take care of the subobject freeness theorems; but there probably is a general theorem, yet to be formulated and proved, that can take care of those.
\end{quote}

The Dniester Notebook question of Parfenov also asked whether there existed Nielsen--Schreier varieties of algebras with one binary operation other than the five varieties mentioned above. This latter question was answered by the second author \cite{MR1302528} who proved the Nielsen--Schreier property for the variety of all algebras satisfying the identity $xx^2=0$, which has remained the only other known example beyond the five classical ones. In general, the Nielsen--Schreier property has been perceived as very rare. For instance, it has been known that among all possible varieties of Lie algebras only the variety of all Lie algebras and the variety of Lie algebras with zero Lie bracket are Nielsen--Schreier (Bakhturin \cite{MR241484}), and that for $n>2$ the only Nielsen--Schreier variety of $n$-Lie algebras is the variety of trivial algebras (Khashina \cite{MR1138459}). If one considers more general structure operations, the Nielsen--Schreier property was proved for varieties of algebras with structure operations that satisfy \emph{no identities} (Kurosh \cite{MR0131427}), and for varieties of algebras whose only identities are particular symmetry properties of structure operations under permutations of arguments (Polin~\cite{MR0237404}); more recently, a similar but more complicated result was established by Shestakov and the second author \cite{MR1899864} for the variety of Akivis algebras. The Nielsen--Schreier property is also true for a number of varieties closely related to that of Lie algebras: the varieties of Lie $p$-algebras (Witt \cite{MR77525}), of Lie superalgebras (Mikhalev and Shtern \cite{MR797705,MR847425}), and of Lie $p$-superalgebras (Mikhalev \cite{MR939518}). Finally, there is an elegant observation of Mikhalev and Shestakov \cite{MR3169596} that one can get new Nielsen--Schreier varieties by forming PBW-pairs with known ones; this gives new proofs for many of the above cases, as well as for the variety of Sabinin algebras, first proved to be Nielsen--Schreier in \cite{MR2855109}.

The second author established \cite{MR1302528,MR1399590} that over a field of zero characteristic a variety $\mathfrak{M}$ is Nielsen--Schreier if and only if the following two conditions hold:
\begin{enumerate}
\item for every free $\mathfrak{M}$-algebra $A$, its universal multiplicative enveloping algebra $U_\mathfrak{M}(A)$ is a free associative algebra,
\item for every homogeneous subalgebra $H$ of every free $\mathfrak{M}$-algebra $A$, the universal multiplicative enveloping algebra $U_\mathfrak{M}(A)$ is a free right $U_\mathfrak{M}(H)$-module. 
\end{enumerate}
These conditions are generally very hard to check, so this criterion has not been used to advance in classification of Nielsen--Schreier varieties. In our present work we found a criterion of a completely different flavour that is easy to check and holds in a large number of examples, allowing us to exhibit many new examples of Nielsen--Schreier varieties. Our criterion is expressed in terms of Gr\"obner bases for operads \cite{MR3642294,MR2667136}. 

\begin{itheorem}[{Th.~\ref{th:SchrComb}}]
Suppose that the operad $\calO$ encoding the given variety of algebras $\mathfrak{M}$ satisfies the following two properties:
\begin{itemize}
\item[(M1)] for the reverse graded path-lexicographic ordering, each leading term of the reduced Gr\"obner basis of the corresponding shuffle operad $\calO^f$ has the minimal leaf directly connected to the root,
\item[(M2)] there exists an ordering for which each leading term of the reduced Gr\"obner basis of the corresponding shuffle operad $\calO^f$ is a left comb whose second smallest leaf is a sibling of the minimal leaf.
\end{itemize}
Then the variety $\mathfrak{M}$ has the Nielsen--Schreier property.
\end{itheorem}

Among the new varieties that we show to have the Nielsen--Schreier property using this criterion, the following examples are perhaps the most interesting:
\begin{itemize}
\item the variety of pre-Lie (also known as right-symmetric) algebras,
\item the variety of Lie-admissible algebras, 
\item the variety of nonassociative algebras satisfying the identity \[xx^2+\alpha x^2x=0,\] for every given $\alpha\ne 1$,
\item the variety of nonassociative algebras satisfying the identity \[x(x(\cdots(xx^2)))=0.\] 
\end{itemize}
The last two examples, both of which generalize the identity $xx^2=0$ of \cite{MR1302528}, show that, over a field of zero characteristic, the set of Nielsen--Schreier varieties of algebras with just one binary operation is infinite in two different ways, containing both countable families with growing degrees of identities and parametric families with fixed degrees of identities; this suggests that Nielsen--Schreier varieties are not as rare as they were thought to be. We observe, however, that most of the identities $xx^2+\alpha x^2x=0$ are quasi-equivalent (in the sense of Albert \cite{MR32609}) to the identity $xx^2=0$; to exhibit an infinite family of Nielsen--Schreier varieties which are not related by quasi-equivalence, one has to take a similar family \[\alpha_1 (x^2x)x+\alpha_2(xx^2)x+\alpha_3(x^2)^2+\alpha_4 x(x^2x)+\alpha_5x(xx^2)=0\] of varieties defined by identities of degree four (we show that such variety is Nielsen--Schreier for generic values of $\alpha_1,\ldots,\alpha_5$).  

The proof of our criterion mainly relies on homotopical algebra for operads and their algebras. Our work should be viewed as another step in the programme of applying operadic methods to classical questions about varieties of algebras, in the spirit of the first author's work with Tamaroff \cite{MR4300233} on a functorial criterion for PBW-pairs of varieties. While our methods may look ``foreign'' to the reader whose intuition comes from classical ring theory, a big advantage of them lies in applicability of the wealth of methods not available on the level of algebras. It is also an immediate consequence of our work that the Nielsen--Schreier property for a variety of algebras automatically implies the Nielsen--Schreier property for the corresponding variety of superalgebras; previously, such results had to be established separately, see, for example, \cite{MR2094332,MR797705,MR847425}. All that said, using operads forces us to only work with identities equivalent to multilinear ones, and so we focus on varieties of algebras over a field of zero characteristic; developing systematic methods for treating the Nielsen--Schreier property in positive characteristic remains an open problem. 

We wish to dedicate this paper to the memory of Vyacheslav Alexandrovich Artamonov who passed away in June 2021. Not only had he been interested in Nielsen--Schreier varieties of algebras throughout his mathematical life \cite{ArtRev,MR0292901,MR3739118,MR2021795}, but also his mathematical heritage is strongly connected with the operad theory: while operads were first defined by J.~P.~May in 1971 in his work on iterated loop spaces \cite{MR2177746,MR0420610}, the same notion seems to have been first introduced under a much more technical name of a ``clone of multilinear operations'' in Artamonov's 1969 paper \cite{MR0237408}. 

\medskip

The paper is organized as follows. In Section \ref{sec:recoll}, we give the necessary background, making emphasis on basics of the operad theory for the reader whose intuition comes from the ring theory. In Section \ref{sec:homol}, we prove a homological criterion of freeness for operadic algebras. In Section \ref{sec:combcrit}, we state and prove our combinatorial criterion. Finally, in Sections \ref{sec:examplesfirst}--\ref{sec:exampleslast}, we present numerous applications of our result, exhibiting many new Nielsen--Schreier varieties of algebras.

\section{Conventions and recollections}\label{sec:recoll}

All vector spaces in this paper are defined over a field $\k$ of zero characteristic. We use somewhat freely the language of category theory \cite{MR0354798}, but go into great detail to explain basics of the operad theory to the reader whose intuition comes from the ring theory. Further details are available in the monographs \cite{MR3642294,MR2954392}.

\subsection{Nielsen--Schreier varieties of algebras}\label{sec:vars}

Let $\mathfrak{M}$ be a variety of algebras, that is a class of algebras over $\k$ with certain structure operations satisfying certain identities. We shall assume that the set of basic structure operations of $\mathfrak{M}$ does not include any elements of arity $0$ or $1$ and has finitely many operations of each arity; thus, we shall not consider unital associative algebras or Rota--Baxter type algebras (which are always equipped with a Rota--Baxter operator $R$ of arity one), or vertex algebras (where one has infinitely many binary operations). Since we work over a field of zero characteristic, every system of identities is equivalent to a system of multilinear ones; sometimes we shall use non-multilinear identities for brevity, and we shall freely move between multilinear and non-multilinear descriptions of the same variety.

For a set $X$, we shall denote by $F_\mathfrak{M}\langle X\rangle$ the free $\mathfrak{M}$-algebra generated by $X$. It has a grading with respect to which all elements of $X$ have degree one. For an element $f\in F_\mathfrak{M}\langle X\rangle$, we denote by $\hat{f}$ the nonzero homogeneous component of $f$ of maximal degree. A system of elements $f_1,\ldots,f_p\in F_\mathfrak{M}\langle X\rangle$ is said to be \emph{irreducible} if no element $\hat{f_i}$ belongs to the subalgebra generated by $\hat{f_j}$ with $j\ne i$. Such a system of elements is said to be \emph{algebraically independent} if the obvious map from the free algebra on $p$ generators to the subalgebra these elements generate is an isomorphism. 

\begin{proposition}[{\cite{MR0292901,MR224663}}]
Over an infinite field, the following properties of a variety $\mathfrak{M}$ are equivalent:
\begin{itemize}
\item every irreducible system of elements in every free $\mathfrak{M}$-algebra is algebraically independent,
\item every subalgebra of every free $\mathfrak{M}$-algebra is free. 
\end{itemize}
\end{proposition}

We call a variety $\mathfrak{M}$ satisfying either of these equivalent properties a \emph{Nielsen--Schreier variety}. One interesting consequence of the Nielsen--Schreier property is that, by a striking contrast to the case of polynomial rings, one can give a complete description of the group of automorphisms of a finitely generated free algebra. 

\begin{theorem}[{\cite[Th.~4]{MR224663}}]\label{th:tame}
Let $\mathfrak{M}$ be a Nielsen--Schreier variety of algebras, and let $A=F_\mathfrak{M}\langle x_1,\ldots,x_n\rangle$ be a finitely generated free algebra in this variety. The group of automorphisms of $A$ is generated by the permutations of $x_1,\ldots,x_n$ together with the automorphisms 
 \[
x_i\mapsto  
\begin{cases}
\alpha x_1+w(x_2,\ldots,x_n),\,\,\quad i=1,\\
\quad\qquad x_i,\quad\quad\qquad\qquad i\ne 1.
\end{cases}
 \]
In other words, every automorphism of a finitely generated free $\mathfrak{M}$-algebra is tame.  
\end{theorem}

Moreover, a result of the second author shows that in this case it is possible to describe the automorphism group of each free algebra by generators and relations \cite{MR2331759}. In the case of Lie algebras, one can also use freeness of subalgebras of free Lie algebras to perform certain cohomology computations \cite{MR2384793}, and to establish analogues of the Schreier formula for groups \cite{MR1764887}. It would be interesting to explore similar applications for numerous varieties whose Nielsen--Schreier property is proved in this paper. We hope to address this elsewhere. 

\subsection{A non-criterion of Nielsen--Schreier varieties}\label{sec:Burgin}

In \cite{MR0417028}, Burgin introduced a certain ``property (S)'' of a variety of algebras that he claimed to be equivalent to the Nielsen--Schreier property. Our main criterion bears a certain resemblance with that of Burgin, and we feel that it is not unreasonable to discuss the latter. We shall now demonstrate that the claim of \cite{MR0417028} is not true at face value even under the most lenient interpretation. (There are some doubts about this work already in the \texttt{MathReviews} review \cite{ArtRev} by Artamonov, who however questioned the proof, not the result itself.) 

Following \cite{MR0417028}, we say that a variety of algebras has the property (S) if in every minimal identity (that is, an identity that does not follow from identities of smaller degrees) each variable appears as an argument of the top level operation in at least one of the monomials in the identity. The main result of \cite{MR0417028} asserts that a homogeneous variety has the property (S) if and only if it is a Nielsen--Schreier variety. 

The description of the property (S) is vague enough to lead to some minor issues right away. In particular, in the introduction to \cite{MR0417028}, it is claimed that the variety of (left) Leibniz algebras, that is algebras satisfying the identity
 \[
a_1(a_2a_3)=(a_1a_2)a_3+a_2(a_1a_3),
 \]
satisfies the property (S). (At that point, the name ``Leibniz algebras'' was not yet invented, but the corresponding variety of algebras was studied by Bloh in \cite{MR0193114} under the name ``left D-algebras''.) This claim is not correct: while the identity above does satisfy the combinatorial condition of the property (S), its consequence of the same degree
 \[
(a_1a_2)a_3+(a_2a_1)a_3=0
 \]
obviously fails the corresponding combinatorial condition. (In fact, the variety of Leibniz algebras does not have the Nielsen--Schreier property \cite{MR1603357}.)

The example of Leibniz algebras, however, is not the biggest problem that one encounters: there exists an example of a variety that does not have the Nielsen--Schreier property but obviously satisfies the property (S), whatever interpretation of that property one may choose. It is the variety of \emph{commutative} algebras satisfying the identity 
 \[
(a_1a_2)a_3+(a_2a_3)a_1+(a_3a_1)a_2=0,
 \]
studied in the recent years by many different authors  \cite{MR3247244,MR1489904,MR3598575}. This identity transforms under the action of $S_3$ as the trivial representation, so the space of identities of degree $3$ is one-dimensional, and the combinatorial condition of the property (S) is satisfied; in fact, combinatorially there is no difference between the multilinear mock-Lie identity and the Jacobi identity. However, due to the commutativity of the operation, this identity is equivalent to the nil identity $x^2x=0$. It is well known that these algebras are Jordan and in particular power associative. To show that it is not a Nielsen--Schreier variety, it is enough to note that in every nontrivial subvariety of power associative algebras, the subalgebra of the free algebra on one generator $x$ generated by $x^2$ and $x^3$ is not free. 

\subsection{The language of symmetric operads}

It is well known that over a field of characteristic zero every system of algebraic identities is equivalent to multilinear ones. Let us briefly explain how this leads to the notion of an operad. To a variety of algebras $\mathfrak{M}$ without constants (operations of zero arity), one may associate the datum 
 \[
\calO=\calO_\mathfrak{M}:=\{\calO(n)\}_{n\ge 1},
 \]
where $\calO(n)$ is the $S_n$-module of multilinear elements (that is, elements of multidegree $(1,1,\ldots,1)$) in the free algebra $F_\mathfrak{M}\langle x_1,\ldots,x_n\rangle$. We note that the underlying vector space of each free algebra can be reconstructed from this datum as  
 \[
\calO(V):=\bigoplus_{n\ge 1}\calO(n)\otimes_{\k S_n} V^{\otimes n},
 \]
where $V$ is the vector space spanned by the generators of that algebra.

Each individual free algebra only has its own algebra structure. If we consider all free algebras at the same time, there is something new that emerges: in the language of category theory, we have a \emph{monad}. To see what this means, we regard the assignment to a vector space $V$ the free algebra generated by $V$ as a functor from the category of vector spaces to itself. Nothing prevents us from applying that functor twice, considering $F_\mathfrak{M}\langle F_\mathfrak{M}\langle V\rangle\rangle$; elements of that vector space are all possible substitution schemes of $\mathfrak{M}$-polynomials into each other. 
Of course, there is a canonical linear map
 \[
\tau_V\colon F_\mathfrak{M}\langle F_\mathfrak{M}\langle V\rangle\rangle\to F_\mathfrak{M}\langle V\rangle,
 \]
which says that for a substitution scheme, we can actually perform a substitution, and write an $\mathfrak{M}$-polynomial of $\mathfrak{M}$-polynomials as an $\mathfrak{M}$-polynomial. This map $\tau$ gives our functor a monad structure, meaning that it is associative: if we apply our functor three times, forming the gigantic algebra
 \[
F_\mathfrak{M}\langle F_\mathfrak{M}\langle F_\mathfrak{M}\langle V\rangle\rangle\rangle,
 \]
there are two different maps to $F_\mathfrak{M}\langle V\rangle$, depending on the order of substitutions, and those give the same result. (Strictly speaking, to talk about a monad, one should also discuss the unitality, but the compatibility of the maps $\tau_V$ with the obvious embedding maps $\imath_V\colon V\to F_\mathfrak{M}\langle V\rangle$ is too trivial to spend time on it.) This can be restricted to multilinear elements, and it defines what is called an operad structure on the sequence $\{\calO(n)\}_{n\ge 1}$; the definition is spelled explicitly below. In the language of operads, our above constraints on varieties of algebras are as follows. Absence of structure operations of arity $0$ is described by the word ``reduced'', absence of structure operations of arity $1$ is described by the word ``connected'' (note that there always exists one ``trivial'' operation of arity $1$, which is identical on every element of every algebra; this operation is not regarded as a structure operation and is permitted). Thus, throughout this paper we work with reduced connected operads.

One key difference between the language of varieties and the language of operads may have already become apparent. Namely, in terms of varieties of algebras, properties like commutativity and associativity are on the same ground: both express certain identities in algebras. In terms of operads, commutativity of an operation is a symmetry type, so it is encoded on the level of $S_n$-modules, while associativity of an operation is an identity: a relation between results of substitution of operations into one another. This distinction will be very important for us. A clean interpretation of how an operad structure formalizes the notion of substitutions of multilinear maps uses the language of linear species, which we shall now recall. 

The theory of species of structures originated at the concept of a combinatorial species, invented by Joyal \cite{MR633783} and presented in great detail in \cite{MR1629341}. The same definitions apply if one changes the target symmetric monoidal category; in particular, if one considers the category of vector spaces, one obtains what is called a linear species. Let us recall some key definitions, referring the reader to~\cite{MR2724388} for further information.  

A \emph{linear species} is a contravariant functor from the groupoid of finite sets (the category whose objects are finite sets and whose morphisms are bijections) to the category of vector spaces. This definition is not easy to digest at a first glance, and a reader with intuition coming from varieties of algebras is invited to think of the value $\calS(I)$ of a linear species $\calS$ on a finite set $I$ as of the set of multilinear operations of type $\calS$ (accepting arguments from some vector space $V_1$ and assuming values in some vector space $V_2$) whose inputs are indexed by~$I$. For example, one can define the Lie species whose component $\Lie(I)$ is the vector space of all multilinear Lie polynomials whose arguments are in one-to-one correspondence with the set $I$; similarly, the associative species has as the component $\Ass(I)$ the vector space of all multilinear associative noncommutative polynomials whose arguments are in one-to-one correspondence with~$I$. A linear species $\calS$ is said to be \emph{reduced} if $\calS(\varnothing)=0$; this means that we do not consider ``constant'' multilinear operations. (This is perhaps the only situation where several different terminologies clash in our paper: we use the word ``reduced'' for linear species to indicate that the value on the empty set is zero, and for Gr\"obner bases to indicate that we consider the unique Gr\"obner basis of a certain irreducible form.)

Sometimes, a ``skeletal definition'' is preferable: the category of linear species is equivalent to the category of symmetric sequences $\{\calS(n)\}_{n\ge 0}$, where each $\calS(n)$ is a right $S_n$-module, a morphism between the sequences $\calS_1$ and $\calS_2$ in this category is a sequence of $S_n$-equivariant maps $f_n\colon \calS_1(n)\to\calS_2(n)$. While this definition may seem more appealing, the functorial definition simplifies the definitions of operations on linear species: it is harder to comprehend the two following definitions skeletally.

The \emph{Cauchy product} of two linear species $\calS_1$ and $\calS_2$ is defined by the formula
 \[
(\calS_1\cdot\calS_2)(I):=\bigoplus_{I=I_1\sqcup I_2}\calS_1(I_1)\otimes\calS_2(I_2).
 \]
One may consider monoids with respect to the Cauchy product which are called \emph{twisted associative algebras} \cite{MR3642294}, and are useful when working with universal multiplicative enveloping algebras of algebras in different varieties of algebras; we shall use them meaningfully below. Additionally, the crucial \emph{composition product} of linear species is compactly expressed via the Cauchy product as
 \[
\calS_1\circ\calS_2:=\bigoplus_{n\ge 0}\calS_1(\{1,\ldots,n\})\otimes_{\k S_n}\calS_2^{\cdot n}, 
 \]
that is, if one unwraps the definitions, 
 \[
(\calS_1\circ\calS_2)(I)
=\bigoplus_{n\ge 0}\calS_1(\{1,\ldots,n\})\otimes_{\k S_n}\left(\bigoplus_{I=I_1\sqcup \cdots\sqcup I_n}\calS_2(I_1)\otimes\cdots\otimes \calS_2(I_n)\right).
 \]
The linear species $\mathbbold{1}$ which vanishes on a finite set $I$ unless $|I|=1$, and whose value on $I=\{a\}$ is given by $\k a$ is the unit for the composition product: we have $\mathbbold{1}\circ\calS=\calS\circ\mathbbold{1}=\calS$.

Formally, a \emph{symmetric operad} is a monoid with respect to the composition product. It is just the multilinear version of substitution schemes of free algebras discussed above, but re-packaged in a certain way. The advantage is that existing intuition of monoids and modules over them, available in any monoidal category \cite{MR0354798}, can be used for studying varieties of algebras. In particular, one can talk about left or right modules over operads, a notion which does not emerge too frequently in the context of varieties of algebras (though right ideals of an operad have been extensively studied in the theory of PI-algebras under the name ``T-spaces''). 

The free symmetric operad generated by a linear species $\calX$ is defined as follows. Its underlying linear species is the species $\calT(\calX)$ for which $\calT(\calX)(I)$ is spanned by decorated rooted trees (including the rooted tree without internal vertices and with just one leaf, which corresponds to the unit of the operad): the leaves of a tree must be in bijection with $I$, and each internal vertex $v$ of a tree must be decorated by an element of $\calX(I_v)$, where $I_v$ is the set of incoming edges of $v$. Such decorated trees should be thought of as tensors: they are linear in each vertex decoration. The operad structure is given by grafting of trees onto each other. We remark that if one prefers the skeletal definition, one can talk about the free operad generated by a collection of $S_n$-modules, but the formulas will become heavier.

As an example of a free symmetric operad, let us consider the linear species $\calX$ for which 
 \[
\calX(I)=
\begin{cases}
\k e_a\wedge e_b, I=\{a,b\},\\
\qquad 0, \quad |I|\ne 2.
\end{cases}
 \]
(Essentially, this is the species encoding one binary anticommutative operation.) The free operad $\calT(\calX)$ has its component $\calT(\calX)(I)$ spanned by all binary trees with leaves indexed by $I$ which represent anticommutative nonassociative monomials: exchange of the two subtrees of a tree multiplies it by $-1$; as we said before, symmetries of operations are already encoded on the level of free operads, before any further identities are introduced. 

\subsection{Shuffle operads and Gr\"obner bases}

We shall now recall how to develop a workable theory of normal forms in operads using the theory of Gr\"obner bases developed by the first author and Khoroshkin \cite{MR2667136}. It is important to emphasize that it is in general extremely hard to find convenient normal forms in free algebras for a given variety~$\mathfrak{M}$. However, focusing on multilinear elements simplifies the situation quite drastically: for instance, for a basis in multilinear elements for the operad controlling Lie algebras one may take all left-normed commutators of the form
$[[[a_1,a_{i_2}],\cdots],a_{i_n}]$, 
where $i_2$,\ldots, $i_n$ is a permutation of $2$,\ldots,$n$; by contrast, all known bases in free Lie algebras \cite{MR1231799} are noticeably harder to describe.

To define Gr\"obner bases for operads, one builds, step by step, an analogue of the theory of Gr\"obner bases for noncommutative associative algebras. To do this, one has to abandon the universe that has symmetries, for otherwise one does not have a sufficiently large supply of operads with monomial relations: vanishing of a monomial forces vanishing of all monomials obtained from it by a symmetric group action, see \cite[Prop.~5.2.2.5]{MR3642294}. The kind of monoids that have a good theory of Gr\"obner bases are \emph{shuffle operads}. A rigorous definition of a shuffle operad uses ordered species \cite{MR1629341}, which we shall now discuss in the linear context. 

An \emph{ordered linear species} is a contravariant functor from the groupoid of finite ordered sets (the category whose objects are finite totally ordered sets and whose morphisms are order preserving bijections) to the category of vector spaces. In terms of the intuition with multilinear maps, this more or less corresponds to choosing a basis of multilinear operations whose inputs are indexed by an ordered set $I$. An ordered linear species $\calS$ is said to be \emph{reduced} if $\calS(\varnothing)=0$. 

The \emph{shuffle Cauchy product} of two ordered linear species $\calS_1$ and $\calS_2$ is defined by the same formula as in the symmetric case:
 \[
(\calS_1\cdot_\Sha\calS_2)(I):=\bigoplus_{I=I_1\sqcup I_2}\calS_1(I_1)\otimes\calS_2(I_2).
 \]
One may consider monoids with respect to the Cauchy product; they are called \emph{shuffle algebras} \cite{MR3642294,MR2918719} or \emph{permutads} \cite{MR2995045}. Moreover, even in the absence of symmetric group actions, the extra datum of an order in our category allows one to define \emph{divided powers} of reduced ordered linear species by
 \[
\calS^{(n)}(I):=\bigoplus_{\substack{I=I_1\sqcup \cdots\sqcup I_n,\\ 
I_1,\ldots, I_n\ne\varnothing, \min(I_1)<\cdots<\min(I_n)}}\calS(I_1)\otimes\cdots\otimes \calS(I_n).
 \]
Using those, the \emph{shuffle composition product} of two reduced ordered linear species $\calS_1$ and $\calS_2$ is defined by the formula
 \[
\calS_1\circ_\Sha\calS_2:=\bigoplus_{n\ge 1}\calS_1(\{1,\ldots,n\})\otimes\calS_2^{(n)},
 \]
that is, if one unwraps the definitions, 
 \[
(\calS_1\circ_\Sha\calS_2)(I)=\bigoplus_{n\ge 1}\calS_1(\{1,\ldots,n\})\otimes\left(\bigoplus_{\substack{I=I_1\sqcup \cdots\sqcup I_n,\\ 
I_1,\ldots, I_n\ne\varnothing,\\ \min(I_1)<\cdots<\min(I_n)}}\calS_2(I_1)\otimes\cdots\otimes \calS_2(I_n)\right).
 \]
The linear species $\mathbbold{1}$ discussed above may be regarded as an ordered linear species; as such, it is the unit of the shuffle composition product.

Formally, a \emph{shuffle operad} is a monoid with respect to the shuffle composition product. As we shall see below, each symmetric operad gives rise to a shuffle operad, and that is the main reason to care about shuffle operads. However, we start with explaining how to develop a theory of Gr\"obner bases of ideals in free shuffle operads.

To describe free shuffle operads, we first define shuffle trees. Combinatorially, a \emph{shuffle tree} is a planar rooted tree whose leaves are indexed by a finite ordered set $I$ in such a way that the following ``local increasing condition'' is satisfied: for every vertex of the tree, the minimal leaves of trees grafted at that vertex increase from the left to the right. The free shuffle operad generated by an ordered linear species $\calX$ can be defined as follows. It is an ordered linear species $\calT_\Sha(\calX)$ for which $\calT_\Sha(\calX)(I)$ is spanned by decorated shuffle trees:  each internal vertex $v$ of a tree must be decorated by an element of $\calX(I_v)$, where $I_v$ is the set of incoming edges of $v$, ordered from the left to the right according to the planar structure. Such decorated trees should be thought of as tensors: they are linear in each vertex decoration. The operad structure is given by grafting of trees onto each other. One particular class of shuffle trees we shall consider are the so called \emph{left combs}: trees for which all vertices appear on the unique path from the root to the minimal leaf. For instance, among the shuffle trees
 \[
\lbincomb{}{}{1}{2}{3},\quad \lbincomb{}{}{1}{3}{2}, \quad\rbincomb{}{}{1}{2}{3},\quad \binfork{}{}{}{1}{2}{3}{4}
 \]
the first two are left combs, and the last two are not.

Given a basis of the vector space of an ordered linear species $\calX$, one may consider all shuffle trees whose vertices are decorated by those basis elements. Such shuffle trees with leaves in a bijection with the given ordered set $I$ form a basis of $\calT_\Sha(\calX)(I)$, and we shall think of them as monomials in the free shuffle operad. 

The next step in developing a theory of Gr\"obner bases is to define divisibility of monomials. Suppose that we have a shuffle tree $S$. We can  insert another shuffle tree $S'$ into an internal vertex of $S$, and connect its leaves to the children of that vertex so that the order of leaves agrees with the left-to-right order of the children. We say that the thus obtained shuffle tree is divisible by $S'$, and use this notion of divisibility to define divisibility of decorated shuffle trees, that is of monomials in the free operad. For example, if we work in the free operad generated by the ordered linear species $\calX$ such that $\calX(I)$ is nonzero only for $|I|=2$, and is spanned by one element $x$, we may insert the tree $U=\lbincomb{x}{x}{1}{3}{2}$ into the ternary vertices of the trees
 \[
\ltbcomb{\phantom{\sum\limits_{n=1}^\infty a_n}}{x}{1}{4}{2}{3} \qquad\text{and}\qquad \lbtcomb{x}{\phantom{\sum\limits_{n=1}^\infty a_n}}{1}{3}{4}{2} . 
 \]
We obtain 
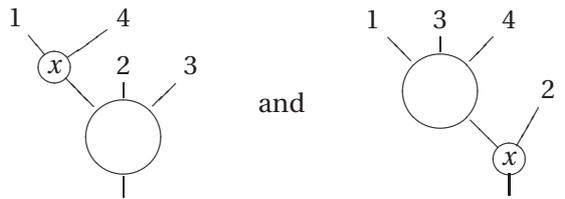
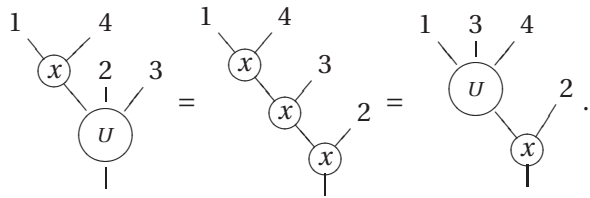
 \[
\ltbcomb{\substack{\phantom{1}\\\phantom{1}U\phantom{1}\\\phantom{1}}}{x}{1}{4}{2}{3}=\llbincomb{x}{x}{x}{1}{4}{3}{2}=\lbtcomb{x}{\substack{\phantom{1}\\\phantom{1}U\phantom{1}\\\phantom{1}}}{1}{3}{4}{2}.
 \]

Once divisibility is understood, the usual Gr\"obner--Shirshov method of computing S-polynomials (in the language of Shirshov, one would say ``compositions'', which has the huge disadvantage in the case of operads where the same word is used to talk about the monoid structure), normal forms, etc. works in the usual way. The only other required ingredient is an \emph{admissible ordering of monomials}, that is a total ordering of shuffle trees with the given set of leaf labels which is compatible with the shuffle operad structure. Such orderings exist, and we invite the reader to consult \cite{MR3642294,MR4114993} for definitions and examples. For us the so called graded path-lexicographic ordering and reverse graded path-lexicographic ordering will be of particular importance. With respect to the former, the trees are first compared by the depth of their leaves, while with respect to the latter, one reverses the comparison with respect to the depth of the leaves (in both cases, leaves are considered one by one in their given order).
Throughout the paper, we say ``the (reverse) graded path-lexicographic ordering'' (with the definite article), though  such an ordering depends on some ordering of generators, which one may choose freely.

Let us note that from the operad point of view, it is possible to unravel the mystery behind the mock-Lie counterexample from Section \ref{sec:Burgin}. It turns out that the mock-Lie identity does not form a Gr\"obner basis of the corresponding operad; and in fact, for one of the possible orderings, the reduced Gr\"obner basis for that operad contains the element 
\begin{multline*}
((a_1a_2)a_3)a_4
+((a_1a_2)a_4)a_3
+((a_1a_3)a_2)a_4
\\ +((a_1a_3)a_4)a_2
+((a_1a_4)a_2)a_3
+((a_1a_4)a_3)a_2
\end{multline*}
which fails the combinatorial condition of the property (S). This suggests that perhaps one can repair Burgin's criterion by re-defining the word ``follow'' in ``does not follow from identities of smaller degrees'' using Gr\"obner bases for operads, and then replacing the property (S) by a combinatorial condition of similar flavour. One possible implementation of this plan leads to the combinatorial criterion proved in Section \ref{sec:combcrit} below.

\subsection{From symmetric operads to shuffle operads}

Note that there is a forgetful functor $\calS\mapsto \calS^f$ from all linear species to ordered linear species; it is defined by the formula $\calS^f(I):=\calS(I^f)$, where $I$ is a finite totally ordered set and $I^f$ is the same set but with the total order ignored. The reason to consider ordered linear species, shuffle algebras and shuffle operads is explained by the following proposition. 

\begin{proposition}[{\cite{MR3642294,MR2667136}}]
For any two linear species $\calS_1$ and $\calS_2$, we have ordered linear species isomorphisms
\begin{gather*}
(\calS_1\cdot\calS_2)^f\cong\calS_1^f\cdot_\Sha\calS_2^f,\\
(\calS_1\circ\calS_2)^f\cong\calS_1^f\circ_\Sha\calS_2^f.
\end{gather*}
In particular, applying the forgetful functor to a twisted associative algebra (a monoid for the Cauchy product) produces a shuffle algebra, and applying a forgetful functor to a reduced symmetric operad gives a shuffle operad. The forgetful functor sends modules over symmetric operads to modules over shuffle operads, ideals to ideals, free symmetric operads to free shuffle operads, etc. 
\end{proposition}

As an example, let us suppose that we consider multilinear operations that one may define starting from one binary operation $a_1,a_2\mapsto [a_1,a_2]$ which is skew-symmetric. Then we work with binary trees (each vertex is either a leaf or has two children), and each binary vertex is now decorated by our only structure operation. Applying the forgetful functor means rewriting each such tree in terms of monomials in the free shuffle operad, using the skew-symmetry of the operation. For instance, the ``Jacobiator'' (the element encoding the Jacobi identity)
 \[
[[a_1,a_2],a_3]+[[a_2,a_3],a_1]+[[a_3,a_1],a_2]
 \]
corresponds to the element
 \[
\lbincomb{[-,-]}{[-,-]}{1}{2}{3}+\lbincomb{[-,-]}{[-,-]}{2}{3}{1}+\lbincomb{[-,-]}{[-,-]}{3}{1}{2}
 \]
in the free symmetric operad, and then to the element
 \[
J:=\lbincomb{[-,-]}{[-,-]}{1}{2}{3}-\rbincomb{[-,-]}{[-,-]}{1}{2}{3}-\lbincomb{[-,-]}{[-,-]}{1}{3}{2}
 \]
in the free shuffle operad. (For an operation $f(a_1,a_2)$ without symmetries, one has to introduce another operation $f^\circ(a_1,a_2):=f(a_2,a_1)$ to perform such rewriting: in the world of shuffle operads, we work with $\k$-modules, not $\k S_n$-modules.) There exists an admissible ordering for which the last of the three monomials in $J$ is the largest one, and in order to compute the Gr\"obner basis, we should form the S-polynomial corresponding to the self-overlap $T$ of that leading monomial with itself, which is  
 \[
\ltbcomb{J}{[-,-]}{1}{4}{2}{3}-\lbtcomb{[-,-]}{J}{1}{3}{4}{2}.
 \]
(If we denote $[-,-]$ by $x$, we note that we saw these monomials when discussing divisibility.) The leading terms in an S-polynomial always cancel, and one needs to check if it has a non-zero reduced form with respect to the existing elements: if it does, that reduced form needs to be adjoined to the existing elements in the course of computing the Gr\"obner basis. In our particular case, the S-polynomial gets reduced to zero, and so the Jacobiator forms a Gr\"obner basis. 

To conclude this section, the forgetful functor from symmetric operads to shuffle operads allows one to go from the universe of ``interesting'' objects (actual varieties of algebras) to the universe of ``manageable'' objects (shuffle operads). Besides the symmetric group actions, it does not really lose any information, and, in particular, if certain properties can be expressed by saying that certain vector spaces are equal to zero, one can prove that in the context of shuffle operads (zero is zero with or without the symmetric group actions). For instance, all results on freeness (of operads or modules over operads) are expressed in terms of vanishing of certain homology groups, as we shall see in the next section. 

\subsection{Bar constructions and freeness of operadic modules}

We shall now recall an important technical tool, a homological criterion of freeness of operadic modules. We refer the reader to \cite{MR2494775} for details on operadic modules and their homotopy theory.

Recall that for an operad $\calO$, its left module $\calL$, and its right module $\calR$, there is a two-sided bar construction $\mathsf{B}_\bullet(\calR,\calO,\calL)$. It is a species of chain complexes. The underlying linear species is spanned by rooted trees where for each tree the root vertex is decorated by an element of $\calR$, the internal vertices whose all children are leaves are decorated by elements of $\calL$, and other internal vertices are decorated by elements of $\calO$. The structure of a chain complex is given by the differential that contracts edges of the tree and uses the operadic composition and the module action maps. For an operad with unit, $\mathsf{B}_\bullet(\calO,\calO,\calO)$ is acyclic (has zero homology in positive degrees and one-dimensional homology in degree zero); moreover, for an augmented operad $\calO$ with the augmentation ideal $\calO_+$, the two-sided bar construction $\mathsf{B}_\bullet(\calO,\calO_+,\calO)$ is acylic. This leads to a free resolution of any left $\calO$-module $\calL$ as 
$$\mathsf{B}_\bullet(\calO,\calO_+,\calO)\circ_\calO\calL\cong\mathsf{B}_\bullet(\calO,\calO_+,\calL).$$
This resolution can be used to prove the following result. (A similar result for right modules is slightly simpler, it was proved and used in \cite{MR4300233}.)

\begin{proposition}\label{prop:Homol}
Let $\calO$ be a (reduced connected) operad, and let $\calL$ be a reduced left $\calO$-module. Then $\calL$ is free as a left $\calO$-module if and only if the positive degree homology of the bar construction $\mathsf{B}_\bullet(\mathbbold{1},\calO_+,\calL)$ vanishes; in the latter case, $\calR$ is generated by the degree zero homology of $\mathsf{B}_\bullet(\mathbbold{1},\calO_+,\calL)$. 
\end{proposition}

\begin{proof}
This immediately follows from the existence and uniqueness up to isomorphism of the minimal free $\calO$-module resolution of~$\calL$, which, for connected operads over a field of characteristic zero, is done similarly to the case of modules over rings in \cite{MR82489}, even though the category of left $\calO$-modules is not abelian.  
\end{proof}

In particular, if $\calL^f$ is free as a left $\calO^f$-module, this means that the vanishing condition of this criterion holds, and we may use that very condition in the universe of symmetric operads to conclude that $\calL$ is free as a left $\calO$-module. This idea, first indicated in \cite{MR3203367}, is one of the key features of our approach.  

\subsection{Universal multiplicative enveloping algebras}\label{sec:deriv}

For any variety of algebras $\mathfrak{M}$ and any $\mathfrak{M}$-algebra $A$, one can define the universal multiplicative enveloping algebra $U_\mathfrak{M}(A)$ to consist of all actions of elements of $A$: formally, one keeps one dedicated slot of an operation open, and inserts elements of $A$ in all other slots. This object has a natural associative algebra structure; moreover, the category of left modules over that algebra is equivalent to the category of $\mathfrak{M}$-bimodules over the algebra $A$. A classical exposition is given in \cite[Sec.~II.7]{MR0251099} and in \cite[\S 3.3]{MR668355}, and a presentation in the language of operads is available in \cite[Sec.~1.6]{MR1301191}. In fact, it is possible to view the universal multiplicative enveloping algebra as the arity one part of the \emph{enveloping operad}, see \cite{MR2527612} and \cite[Chapter 4]{MR2494775}. 

We shall use the viewpoint on universal multiplicative enveloping algebras and on enveloping operads that encodes them via particular right operadic modules, mainly following \cite[Chapter 10]{MR2494775}, but slightly re-casting it in the language of linear species. 

Let us recall that the \emph{derivative} $\partial(\calS)$ of a species $\calS$ is defined by the formula
 \[
\partial(\calS)(I):=\calS(I\sqcup\{\star\}),
 \]
so that in the case of linear species of multilinear operations of some type, it forms multilinear operations with one extra dedicated input that does not mix with the others (this is denoted by $\calS[1]$ in \cite{MR2494775}, clashing with one of the usual notations for suspensions; to avoid any confusion, we choose to stick with the more standard notation of \cite{MR1629341}). In the view of the above discussion of multiplicative universal envelopes, this is very appropriate. If $\calO$ is an operad, then $\partial(\calO)$ has two structures: it is a right $\calO$-module (via substitutions into the non-dedicated input) and a \emph{twisted associative algebra}, that is a left module over the associative operad (via concatenating operations, substituting them into the dedicated inputs of one another). These two structures commute: $\partial(\calO)$ is a twisted associative algebra in the symmetric monoidal category of right $\calO$-modules. The universal enveloping algebra is obtained from this via a relative composite product construction \cite{MR2494775,MR1854112}:
 \[
U_{\calO}(A)\cong \partial(\calO)\circ_{\calO} A.
 \]  
This has been used in a crucial way in \cite{MR4381941} to establish the following result: there is a PBW type theorem for universal multiplicative enveloping algebras over the given operad $\calO$ if and only if $\partial(\calO)$ is free as a right $\calO$-module. Moreover, in this case, for the linear species $\calY$ that freely generates that right module we have an isomorphism 
 \[
U_{\calO}(A)\cong\calY(A) 
 \]
that is functorial with respect to $\calO$-algebra morphisms.

Similarly to the passage from symmetric operads to shuffle operads, it is possible to pass from left modules over the associative operad to a certain shuffle version. We shall not discuss this passage in detail, but rather briefly explain what this means for $\partial(\calO)$. In the symmetric context, the operation $\partial$ makes one of the inputs of the operation ``special''. Once we apply the forgetful functor to ordered linear species, the inputs are linearly ordered. A good ``canonical'' way to make one of them special is take the first one. We shall recall the corresponding construction, referring the reader to \cite{MR4381941} for a slightly different viewpoint. For a shuffle operad $\calP$, we let $\partial\Sha(\calP)(I)=\calP(\{-\infty_I\}\sqcup I)$, where $-\infty_I$ denotes a new element that is smaller than all elements of $I$. For each ordered set $K$ partitioned as $K=I\sqcup J$, the product
 \[
\mu_{I,J}\colon\partial^\Sha(\calP)(I)\otimes\partial^\Sha(\calP)(J)\to\partial^\Sha(\calP)(K),
 \]
is defined as follows. Suppose that 
 \[
\alpha\otimes\beta\in \partial^\Sha(\calP)(I)\otimes\partial^\Sha(\calP)(J)=\calP(\{-\infty_I\}\sqcup I)\otimes\calP(\{-\infty_J\}\sqcup J).
 \]
We set 
 \[
\mu_{I,J}(\alpha\otimes\beta)=\alpha\circ_{-\infty_I}\beta.
 \]
(Strictly speaking, this way one obtains an element in $\calP(\{-\infty_J\}\sqcup K)$, and one has to use the unique bijection of ordered sets to land in $\calP(\{-\infty_K\}\sqcup K)$.) We also obtain a designated element in $\partial^\Sha(\calP)(\varnothing)=\calP(\{-\infty_\varnothing\})$ corresponding to the operadic unit.

\begin{proposition}\label{prop:twoder}
For each shuffle operad $\calP$, the thus defined map 
 \[
\partial^\Sha(\calP)\cdot_\Sha\partial^\Sha(\calP)\to \partial^\Sha(\calP)
 \]
satisfies the associativity axiom and the unitality axiom with respect to the designated elements, so $\partial^\Sha(\calP)$ becomes a shuffle algebra. Moreover, if the shuffle operad $\calP$ is of the form $\calP=\calO^f$, where $\calO$ is a symmetric operad, then the shuffle algebra $\partial^\Sha(\calP)$ is isomorphic to $\partial(\calO)^f$.
\end{proposition}

\begin{proof}
Both the associativity and the unitality immediately follow from the corresponding axioms for shuffle operads. (Note that since we compose only at minima, only the sequential axiom of the operad, corresponding to the genuine associativity, will be used.) The statement about the forgetful functor is essentially tautological and holds by direct inspection.
\end{proof}

\subsection{Modules of K\"ahler differentials for algebras over operads}\label{sec:kahler}

Generalizing the classical definition for commutative algebras, one may define the $U_{\calO}(A)$-module of K\"ahler differentials $\Omega^1_A$ for any given operad $\calO$ and any given $\calO$-algebra $A$ \cite[Sec.~4.4]{MR2494775}. 
(We use the notation $\Omega^1_A$ instead of $\Omega^1_{\calO}(A)$ used in \cite{MR2494775} since the latter looks like the evaluation of some linear species $\Omega^1_\calO$ on the underlying vector space of~$A$; in reality, as we shall see below, it is obtained via the \emph{relative} composite product.) The intrinsic definition states that $\Omega^1_A$ is the $U_{\calO}(A)$-module that represents the functor of derivations $\mathrm{Der}(A,E)$ with values in a $U_{\calO}(A)$-module $E$. It is known that this module can be constructed in a very explicit way: it is spanned by formal expressions $p(a_1,\ldots,da_i,\ldots,a_m)$, where $p\in\calO(m)$, and $a_1, \ldots, a_m\in A$, modulo the relations
\begin{multline*}
p(a_1,\ldots,q(a_i,\ldots,a_{i+n-1}),\ldots, da_j,\ldots, a_{m+n-1})=\\
(p\circ_i q)(a_1,\ldots,a_i,\ldots,a_{i+n-1},\ldots, da_j,\ldots, a_{m+n-1}),
\end{multline*}
\begin{multline*}
p(a_1,\ldots,dq(a_i,\ldots,a_{i+n-1}),\ldots, a_{m+n-1})=\\
\sum_{j=i}^{i+n-1}(p\circ_i q)(a_1,\ldots,a_i,\ldots,da_j,\ldots,a_{i+n-1},\ldots, a_{m+n-1}).
\end{multline*}
In these formulas, we rewrite the elements $p(a_1,\ldots,da_i,\ldots,a_m)$ for which one of $a_1$, \ldots, $a_m$ is computed using the $\calO$-algebra structure. However, according to \cite[Sec.~10.3]{MR2494775}, one may also view $q$ in these formulas as the right $\calO$-module action on $p(a_1,\ldots,da_i,\ldots,a_m)$, and observe that, similarly to the universal multiplicative enveloping algebras, there exists a universal right $\calO$-module of K\"ahler differentials $\Omega^1_\calO$ such that
 \[
\Omega^1_A\cong \Omega^1_\calO\circ_{\calO} A.
 \]
This has been used in \cite{DFC} to establish the following result: there is a PBW type theorem for modules of K\"ahler differentials of algebras over the given operad $\calO$ if and only if $\Omega^1_\calO$ is free as a right $\calO$-module. Moreover, in this case, for the linear species $\calZ$ that freely generates that right module we have an isomorphism 
 \[
\Omega^1_A\cong\calZ(A) 
 \]
that is functorial with respect to $\calO$-algebra morphisms. We note that the PBW property for the modules of K\"ahler differentials is discussed in \cite{MR2775896} where it is erroneously claimed that it follows from the PBW property for universal multiplicative enveloping algebras; this fails in many cases, starting from the operad of commutative associative algebras.

\section{Homological criterion of freeness}\label{sec:homol}

Our strategy for establishing freeness of subalgebras is to use homology of algebras over operads. In general, for a given operad, one can define \cite{MR2494775,MR2775896} the Andr\'e--Quillen homology theory for algebras over that operad using the appropriate model category structure. Specifically, for every $\calO$-algebra $A$, one may give the following definitions \cite[Sec.~13.1]{MR2494775}. The homology of $A$ with coefficients in a right $U_{\calO}(A)$-module $F$ is defined as  
 \[
H^{\calO}_\bullet(A,F)\cong H_\bullet(F\otimes_{U_{\calO}(Q_A)}\Omega^1_{Q_A}). 
 \]   
where $Q_A$ is any cofibrant replacement of $A$ in the category of $\calO$-algebras (a reader not fluent in the language of model categories can replace these words by ``free resolution in the category of differential graded $\calO$-algebras'' to have a clear picture of what is going on), and where $F$ is given a structure of a right $Q_A$-module via the augmentation $Q_A\to A$. We shall use homology of algebras over operads to give a criterion of freeness for an algebra $A$ over a given operad $\calO$. 

\begin{proposition}\label{prop:free-alg-homol}
Let $\calO$ be a (reduced connected) operad, and let $A$ be an $\calO$-algebra that possesses an additional positive grading for which all operations of $\calO$ are homogeneous; in other words, $A=\bigoplus_{k>0}A_k$, and $\alpha(A_{i_1},\ldots,A_{i_p})\subset A_{i_1+\cdots+i_p}$ for every structure operation $\alpha$. The $\calO$-algebra $A$ is free if and only if 
for each right $U_{\calO}(A)$-module $F$ we have $H^{\calO}_k(A,F)=0$ for $k>0$.
\end{proposition}

\begin{proof}
We first note that, since a free algebra is its own cofibrant replacement, freeness of an algebra trivially implies that the positive degree homology and cohomology with any coefficients vanishes. 

To prove that homology vanishing ensures freeness, we are going to make several general observations. First, the assumption on positive grading of the algebra $A$ and on connectedness of the operad $\calO$ implies, by an easy induction, that the algebra $A$ possesses a \emph{minimal model}, that is a cofibrant replacement $Q_A$ for which the underlying algebra of $Q_A$ is a free algebra generated by a homologically graded vector space $X_\bullet$ that has an additional internal grading corresponding to that of $A$, and the differential sends every generator $x\in X_\bullet$ to a combination of decomposable elements (those obtained by applying nontrivial operations from $\calO$ to generators). Second, we note that for a free algebra $\calO(X_\bullet)$, we have a vector space isomorphism
 \[
\Omega^1_{\calO(X_\bullet)}\cong U_{\calO}(\calO(X_\bullet))\otimes X_\bullet. 
 \]  
Third, since the operad $\calO$ is connected, it has an augmentation $\calO\to\k$ leading to an augmentation for the universal enveloping algebra $U_{\calO}(A)$; hence, it makes sense to talk about the trivial right $U_{\calO}(A)$-module $\k$, which will be the protagonist of our argument. 

Suppose that for each right $U_{\calO}(A)$-module $F$ we have $H^{\calO}_k(A,F)=0$ for $k>0$. From the observations we made, it follows that for any cofibrant replacement $Q_A$ whose underlying algebra is a free algebra $\calO(X_\bullet)$, we have a vector space isomorphism
 \[
F\otimes_{U_{\calO}(Q_A)}\Omega^1_{Q_A}\cong F \otimes_{U_{\calO}(\calO(X_\bullet))} (U_{\calO}(\calO(X_\bullet))\otimes X_\bullet)\cong F \otimes X_\bullet.
 \]
If $Q_A$ is the minimal model and $F$ is the trivial right module, the differential of this chain complex is immediately seen to vanish. Consequently, we have  
 \[
H^{\calO}_\bullet(A,\k)\cong X_\bullet.
 \]
It follows from our assumption that $X_k$ vanishes for $k>0$, so the minimal model is just $\calO(X_0)$; this vector space is concentrated in homological degree zero, so there is no room for the differential, and $A\cong \calO(X_0)$ is a free $\calO$-algebra. 
\end{proof}

\begin{remark}\leavevmode
\begin{enumerate}
\item In the classical examples of associative algebras and of Lie algebras, the freeness criterion requires that the homology vanishes in degree greater than \emph{one}. This mismatch is explained by noting that in both cases the homology can be computed as the Tor functor over the universal enveloping algebra between the trivial right module and the left module of K\"ahler differentials; however, the classical definition of the Chevalley--Eilenberg and of Hochschild homologies realises them as the Tor functor between the trivial module and the quotient of a free module by its submodule isomorphic to the module of K\"ahler differentials, and the long exact sequence of homology explains the shift. 
\item We insist on stating our homological criterion of freeness for \emph{positively graded} algebras, since this is the best thing to hope within the reach of current knowledge: even for Lie algebras freeness of algebras of homological dimension one is an open question in characteristic zero and is false in positive characteristic, as shown by Mikhalev, Zolotykh and the second author \cite{MR1443453}.
\item Proposition \ref{prop:Homol} can be shown to be a particular case of Proposition \ref{prop:free-alg-homol}; however, to make the exposition more user-friendly for a reader without much experience of operads, we chose to present the two results separately. 
\end{enumerate}
\end{remark}

\section{A combinatorial criterion for the Nielsen--Schreier property}\label{sec:combcrit}

In this section, we prove the main theoretical result of this paper: a combinatorial criterion for the Nielsen--Schreier property in terms of Gr\"obner bases for operads. One part of it, Lemma \ref{lm:freepartial} is of independent interest, for it gives a criterion for the universal multiplicative enveloping algebra to be free as associative algebra, which is likely to be sufficient for a variety to have the Nielsen--Schreier property (see Conjecture \ref{conj:SchrComb} below). 

\begin{theorem}\label{th:SchrComb}
Suppose that the operad $\calO$ encoding the given variety of algebras $\mathfrak{M}$ satisfies the following two properties:
\begin{itemize}
\item[(M1)] for the reverse graded path-lexicographic ordering, each leading term of the reduced Gr\"obner basis of the corresponding shuffle operad $\calO^f$ has the minimal leaf directly connected to the root,
\item[(M2)] there exists an ordering for which each leading term of the reduced Gr\"obner basis of the corresponding shuffle operad $\calO^f$ is a left comb whose second smallest leaf is a sibling of the minimal leaf.
\end{itemize}
Then the variety $\mathfrak{M}$ has the Nielsen--Schreier property.
\end{theorem}

Before proving this theorem, let us give an example of how it can be applied. Consider the operad encoding Lie algebras. It is immediate to check that the shuffle Jacobi identity
 \[
\lbincomb{[-,-]}{[-,-]}{1}{2}{3}-\lbincomb{[-,-]}{[-,-]}{1}{3}{2}-\rbincomb{[-,-]}{[-,-]}{1}{2}{3}
 \]
forms the reduced Gr\"obner basis for both orderings; its last monomial is the leading term for the reverse graded path-lexicographic ordering, and its first monomial is the leading term for the graded path-lexicographic ordering, and the combinatorial conditions of our theorem clearly hold. Thus, our result in particular gives a new one-line proof of the Shirshov--Witt theorem on subalgebras of free Lie algebras. 

\begin{proof}
Let us begin with outlining the general plan of the proof. We shall use the first combinatorial condition to prove that for every free $\mathfrak{M}$-algebra $A$, its universal multiplicative enveloping algebra $U_\mathfrak{M}(A)$ is a free associative algebra. We shall then use one half of the second combinatorial condition (the left comb condition) to prove that the variety $\mathfrak{M}$ has the PBW property for universal multiplicative enveloping algebras. Combining the two previous results, we shall show that the variety $\mathfrak{M}$ has the PBW property for K\"ahler differentials of algebras. Next, we shall use the other half of the second combinatorial condition (the maximal leaf condition) to prove that for every $\mathfrak{M}$-algebra $A$ with zero structure operations and every subspace $H\subset A$, the universal multiplicative enveloping algebra $U_\mathfrak{M}(A)$ is a free $U_\mathfrak{M}(H)$-module. Finally, we shall combine the results we proved to show, using the homological criterion of freeness, that subalgebras of free algebras are free. 

\begin{lemma}\label{lm:freepartial}
Suppose that $\calO$ is an operad satisfying Property (M1). For every free $\calO$-algebra $A$, the universal multiplicative enveloping algebra $U_\calO(A)$ is a free associative algebra.
\end{lemma} 

\begin{proof}
For the free algebra $A=\calO(V)$, we have 
 \[
U_\calO(A)\cong\partial(\calO)\circ_{\calO} A=\partial(\calO)\circ_{\calO} \calO(V)\cong\partial(\calO)(V), 
 \] 
with the product of $U_\calO(A)$ induced from that of $\partial(\calO)$ on the twisted associative algebra level, so it is enough to show that $\partial(\calO)$ is free as a twisted associative algebra. To establish that, it is sufficient to consider the associated shuffle operad $\calO^f$ and prove the corresponding result for the shuffle algebra $\partial^\Sha(\calO^f)$. Indeed, we are working with connected operads over a field of characteristic zero, so the homological criterion of freeness implies that $\partial(\calO)$ is a free twisted associative algebra if and only if $\partial(\calO)^f$ is a free shuffle algebra, and we know that $\partial(\calO)^f\cong\partial^\Sha(\calO^f)$.

Let us denote by $G_r$ the reduced Gr\"obner basis of the shuffle operad $\calO^f$ for the reverse graded path-lexicographic ordering, and by $\calN_r$ the ordered species of monomials that are normal with respect to $G_r$. From the definition of the shuffle algebra structure on $\partial^\Sha(\calO^f)$, it immediately follows that it is generated by ``min-indecomposable'' elements of $\calN_r$, that is shuffle trees which have their minimal leaf directly connected to the root. Moreover, if $\alpha\in\partial^\Sha(\calO^f)(I)$ and $\beta\in\partial^\Sha(\calO^f)(J)$ are two normal forms, then $\mu_{I,J}(\alpha\otimes\beta)$ is a normal form. Indeed,  composing two normal forms at the minimal leaf of a normal form cannot create an element divisible by a leading term of the Gr\"obner basis, by our assumption on the leading terms. 
\end{proof}

The following result follows from \cite[Th.~5.16]{MR4381941}, but we give a proof to make the exposition self-contained. 

\begin{lemma}\label{lm:partialfree}
Suppose that $\calO$ is an operad satisfying the first part of Property (M2): for the graded path-lexicographic ordering, each leading term of the reduced Gr\"obner basis of the corresponding shuffle operad $\calO^f$ is a left comb. The right $\calO$-module $\partial(\calO)$ is free. Accordingly, there is a PBW-type theorem for multiplicative universal envelopes of $\calO$-algebras: there exists a linear species $\calY$ such that for every $\calO$-algebra $A$, the underlying vector space of $U_\calO(V)$ is isomorphic to $\calY(A)$ functorially with respect to $\calO$-algebra morphisms.
\end{lemma}

\begin{proof}
Let us denote by $G_l$ the reduced Gr\"obner basis of the shuffle operad $\calO^f$ for the graded path-lexicographic ordering, by $\calN_l$ the ordered species of monomials that are normal with respect to $G_l$, and by $\calN^{(0)}_l\subset\calN_l$ the ordered subspecies of normal monomials that are left combs. It is clear that the right $\calO^f$-module $\partial^\Sha(\calO^f)$ is freely generated by $\partial^\Sha(\calN^{(0)}_l)$. Indeed, this follows from the fact that composing two normal forms at a non-minimal leaf cannot create an element divisible by a leading term of the Gr\"obner basis, by our assumption on the leading terms. Since we are working with connected operads over a field of characteristic zero, the homological criterion of freeness implies that $\partial(\calO)$ is free as a right $\calO$-module, which in turn implies that the operad $\calO$ has the PBW property for universal multiplicative enveloping algebras.
\end{proof}

The next result is interesting in its own rights, as it shows how to apply Gr\"obner bases for operads to study modules of K\"ahler differentials. We warn the reader that the claim in \cite{MR2775896} that the PBW-type theorem for K\"ahler differentials is available whenever the PBW theorem for universal envelopes is available is false already for commutative associative algebras.

\begin{lemma}\label{lm:kahlerfree}
Suppose that $\calO$ is an operad satisfying Property (M1) and the first part of Property (M2). The right $\calO$-module $\Omega^1_\calO$ is free. Accordingly, there is a PBW-type theorem for modules of K\"ahler differentials of $\calO$-algebras: there exists a linear species $\calZ$ such that for every $\calO$-algebra $A$, the underlying vector space of $\Omega^1_A$ is isomorphic to $\calZ(A)$ functorially with respect to $\calO$-algebra morphisms.
\end{lemma}

\begin{proof}
Let us consider the operad $\calO_{\mathrm{Der}}$ encoding the datum of an $\calO$-algebra $A$ equipped with a degree zero derivation $d$. This operad has a weight grading for which all elements of $\calO$ are of weight $0$ and the element $d$ is of weight $1$; the elements of weight $1$ in $\calO_{\mathrm{Der}}$ form a right $\calO$-module that is obviously isomorphic to $\Omega^1_\calO$. Thus, we may study the module $\Omega^1_\calO$ using a suitable Gr\"obner basis for the operad $\calO_{\mathrm{Der}}$. 

Let us use the reverse graded path-lexicographic ordering for the operad $\calO_{\mathrm{Der}}^f$. Since we now use elements of arity one, this ordering at a first glance does not satisfy the descending chain condition; however, the presence of the extra weight grading mentioned above repairs this issue, since each weight graded component is finite-dimensional. If we focus on elements on weight $0$, we are working with the operad $\calO$, and, as in the proof of Lemma \ref{lm:freepartial}, we denote by $G_r$ the reduced Gr\"obner basis of the shuffle operad $\calO^f$ for the ordering we consider. To incorporate the derivation condition, we have to impose, for each generator $\alpha\in\calO(k)$, the derivation relation
$d\circ_1\alpha=\sum_{i=1}^k\alpha\circ_i d$.  
For the reverse graded path-lexicographic ordering, the leading term of this relation is $\alpha\circ_k d$. Let us show that the monomials that are normal with respect to $G_r$ and with respect to the derivation relations form a basis of the operad $\calO_{\mathrm{Der}}^f$. For that, we note that these normal monomials are precisely the normal monomials of $\calO^f$ with additional decorations: each edge of the shuffle tree, whether internal or external, is decorated with some power $d^s$, unless this edge is the last input edge of some internal vertex, in which case it cannot have any decoration. Thus, the decoration space of a shuffle tree with $p$ edges and $s$ internal vertices is $\k[d]^{\otimes (p-s)}$, and the generating function of dimensions of weight graded components of the vector space spanned by these elements is $\frac{1}{(1-t)^{p-s}}$. At the same time, it is well known that $\calO_{\mathrm{Der}}\cong \calO\circ\k[d]$: for each basis element of $\calO$ of arity $a$, one may decorate each of its $a$ leaves by a power of $d$ and obtain a basis, so we get a vector space with the generating function of dimensions of weight graded components $\frac{1}{(1-t)^{a}}$. It remains to notice that $a=p-s$, since this property is true for trees with one internal vertex, and is preserved under composition. Therefore, the monomials that are normal with respect to $G_r$ and with respect to the derivation relations (which necessarily form a spanning set of $\calO_{\mathrm{Der}}^f$) span a vector space whose generating function of dimensions of weight graded components is equal to that of $\calO_{\mathrm{Der}}$, so they form a basis of $\calO_{\mathrm{Der}}$, and thus if we adjoin to $G_r$ the derivation relations, we get a Gr\"obner basis. 

Let us examine the normal monomials of weight one with respect to the Gr\"obner basis that we found (as we mentioned above, the space of elements of weight one form a right $\calO$-module isomorphic to $\Omega^1_\calO$). Each such monomial is obtained by grafting of a normal monomial $T''\in\calO(I'')$ onto the only child labelled $\star$ of a vertex labelled $d$ of a certain normal monomial $T'\in\calO(I'\sqcup\{\star\})$. We note that if we fix $I''$ and change the set $I'$, the set of possible monomials $T'$ is identified with $\partial(\calO)$. According to Lemma \ref{lm:partialfree}, the right $\calO$-module $\partial(\calO)$ is free, and the right action of $\calO$ on $\Omega^1_\calO$ combines the right action $\partial(\calO)$ with the composition at the leaf labelled $\star$, so we conclude that this right module is also free.  
\end{proof}

The next result is perhaps the least obvious of the key steps of the proof, as it relies on an intricate mix of Gr\"obner bases for operads and for shuffle algebras.

\begin{lemma}\label{lm:freemodule}
Suppose that $\calO$ is an operad satisfying Property (M2). For every $\calO$-algebra $A$ with zero operations and any subspace $H\subset A$, viewed as a subalgebra with zero operations, the universal enveloping algebra $U_\calO(A)$ is free as a right $U_\calO(H)$-module.
\end{lemma}

\begin{proof}
According to Lemma \ref{lm:partialfree}, $\partial(\calO)\cong\calY\circ\calO$ for some linear species $\calY$. To comprehend universal multiplicative envelopes for algebras with zero operations functorially, it is convenient to think of such an algebra $A$ as $\mathbbold{1}(A)$, where $\mathbbold{1}$ is given a left $\calO$-module structure via the augmentation map. This way, 
 \[
\partial(\calO)\circ_\calO A=\partial(\calO)\circ_\calO (\mathbbold{1}\circ A)\cong (\partial(\calO)\circ_\calO \mathbbold{1})\circ A,
 \]
and we see that the suitable formula for the species $\calY$ is \[\calY:=\partial(\calO)\circ_\calO\mathbbold{1}; \] this way, $\calY(A)$ literally corresponds to the universal envelope of the algebra with zero operations. The advantage of this viewpoint is that $\calY$ is, by construction, a twisted associative algebra, and the associative algebra structure of $U_\calO(A)\cong\calY(A)$ is induced from that twisted associative algebra structure. 

From the formula $\calY=\partial(\calO)\circ_\calO\mathbbold{1}$ it follows that any right action of a nontrivial structure operation vanishes, so as a twisted associative algebra, $\calY$ is generated by $\partial(\calX)$, where $\calX$ is the species of generators of $\calO$, and therefore as a shuffle algebra, $\calY^f$ is generated by $\partial(\calX)^f$. Let us describe a Gr\"obner basis of relations for this shuffle algebra. To the Gr\"obner basis $G_l$, we may associate the subset $\bar{G}_l$ of the free shuffle algebra generated by $\partial(\calX)^f$ consisting of elements obtained from elements of $G_l$ by deleting all monomials that are not left combs. It is clear that $\bar{G}_l$ consists of relations of $\calY^f$, since the deleted monomials vanish in $\partial(\calO)\circ_\calO\mathbbold{1}$. Moreover, according to our assumption about the leading terms of $G_l$, the elements of $\bar{G}_l$ have the same leading terms for the appropriate graded lexicographic order of monomials in the free shuffle algebra.
This immediately implies that $\bar{G}_l$ forms a Gr\"obner basis: the cosets of elements from $\partial^\Sha(\calN^{(0)}_l)$ form a basis of $\calY^f$, and these are precisely the normal forms with respect to $\bar{G}_l$. 

The last step of the proof requires to slightly extend the combinatorics with which we work. We shall need the language of two-sorted linear species. Informally, a two-sorted linear species is a canonical rule to associate a vector space to each pair of finite sets, see \cite{MR1629341} for details. In our case, our goal is to prove that for an algebra $A$ with zero operations and its subspace $H$, viewed as a subalgebra with zero operations, $U_\calO(A)$ is free as a $U_\calO(H)$-module. We write $A=H\oplus H'$ for some subspace $H'$, and we wish to make this splitting propagate in a certain way to universal enveloping algebras. The universal enveloping algebra $U_\calO(A)$ can be calculated as $\calY(A)$, and now we shall use two-sorted species to distinguish between elements coming from $H$ and from $H'$. 

To be precise, we consider the two-sorted species $\calY^{(2)}$ and $\calX^{(2)}$ defined as 
 \[
\calY^{(2)}(I,J)=\calY(\{\star\}\sqcup I\sqcup J),
\quad \calX^{(2)}(I,J)=\calX(\{\star\}\sqcup I\sqcup J)
 \]  
The meaning of these species is as follows. The species $\calY^{(2)}$ is a functorial version of the universal multiplicative envelope $U_\calO(A)$ when written as $U_\calO(H\oplus H')$. The species $\calX^{(2)}$ is the functorial species of generators of that algebra: we consider structure operations of our algebras for which we have a special input that will be used to define the twisted associative algebra structure, some inputs of the first type (where we shall later substitute elements of $H$), and some inputs of the second type (where we shall later substitute elements of $H'$). We have $\calX^{(2)}=\calX^{(2)}_0\oplus\calX^{(2)}_1$, where 
 \[
\calX^{(2)}_0(I,J)=
\begin{cases}
\calX^{(2)}(I,J),\,\,\, J=\varnothing,\\
\quad\quad 0,\quad\quad  J\ne\varnothing,
\end{cases}
\qquad
\calX^{(2)}_1(I,J)=
\begin{cases}
\quad\quad 0,\quad\quad J=\varnothing,\\
\calX^{(2)}(I,J),\,\,\, J\ne\varnothing.
\end{cases}
 \]
In plain words, $\calX^{(2)}_0$ will later correspond to the situation where all elements we use are elements of $H$, and $\calX^{(2)}_1(I,J)$ will later correspond to the situation where we use at least one element of $H'$. The subalgebra $U_\calO(H)$ of $U_\calO(A)$ corresponds to the subalgebra $\calY_0^{(2)}$ of $\calY^{(2)}$ generated by $\calX^{(2)}_0$. Thus, we see that it is sufficient to prove that the twisted associative algebra $\calY^{(2)}$ is free as a right $\calY_0^{(2)}$-module: the species of generators of that module, once evaluated on $(H,H')$, will give the free generators of $U_\calO(A)$ as a right $U_\calO(H)$-module. 

We shall consider the two-sorted species in the shuffle context as follows: we consider pairs of ordered sets $(I,J)$ as totally ordered sets for which all elements of $J$ are smaller than all elements of $I$. Let us examine closely a normal form from $\calN^{(0)}_l(I,J)$. Such normal forms, which are, by definition, left combs, come in two types: those normal forms for which all siblings of the minimal leaf belong to $I$, and those normal forms which have a leaf from $J$ among the siblings of the minimal leaf. We claim that the normal forms of the second kind are free generators of $\calY^{(2)}$ as a right $\calY_0^{(2)}$-module. Indeed, every normal form can be represented as a shuffle product of a normal form of the second kind we just described with several generators from $\calX^{(2)}_0$, so those normal forms generate $\calY^{(2)}$ as a right $\calY_0^{(2)}$-module. Freeness follows from our assumption on the leading terms: a shuffle product of a normal form from $\calY_0^{(2)}$ and a normal form of the second kind cannot be divisible by a leading term of the Gr\"obner basis.  
\end{proof}

To establish the Nielsen--Schreier property, we have to prove freeness of arbitrary subalgebras of free algebras. We shall now show that this would follow from freeness of a very particular system of subalgebras. Our approach is a generalization of the trick of Shirshov used in the case of Lie algebras \cite{MR0059892}. Namely, let us consider the free $\mathfrak{M}$-algebra $F_n$ with generators $y,x_1,\ldots,x_n$. This algebra admits an obvious homomorphism to the one-dimensional vector space spanned by $y$, viewed as an $\mathfrak{M}$-algebra with zero structure operations. Let us denote by $H_n$ the kernel of that homomorphism, viewed as an $\mathfrak{M}$-subalgebra of the free algebra$F_n$. 

\begin{lemma}\label{lm:shirshov-trick}
If the $\mathfrak{M}$-algebra $H_n$ is free for each $n\ge 0$, the variety $\mathfrak{M}$ has the Nielsen--Schreier property. 
\end{lemma}

\begin{proof}
Let $H$ be a subalgebra of the free $\mathfrak{M}$-algebra $F$ with $k$ generators. Without loss of generality, we may assume $H$ to be finitely generated by an irreducible system of elements $h_1,\ldots,h_q$. We shall prove our claim by Noetherian induction on $(\deg(h_1),\ldots,\deg(h_q))\in\mathbb{N}^q$, where we use the partial order of $\mathbb{N}^q$ for which $\mathbf{u}<\mathbf{v}$ if and only if $\mathbf{u}\ne\mathbf{v}$ and all coordinates of $\mathbf{v}-\mathbf{u}$ are nonnegative. The basis of induction is the case of all degrees equal to one, in which case we consider a subalgebra generated by several generators, and the claim is clear. Otherwise, the system of elements $h_1,\ldots,h_q$ cannot contain all generators of $F$; we may assume that it does not contain the generator $x_k$, but that $x_k$ nontrivially appears in some of the $\mathfrak{M}$-monomials used to define the elements $h_1,\ldots,h_q$, for otherwise we could find a counterexample in a smaller free algebra. If we denote $x_k$ by $y$, thus identifying our ambient free algebra $F$ with the free algebra generated by $y$, $x_1$,\ldots, $x_{k-1}$, we see that all elements $h_1,\ldots,h_q$ belong to $H_{k-1}$. By our assumption, this algebra is free; moreover, each element $h_i$ for which $y$ nontrivially appears in some of the $\mathfrak{M}$-monomials has smaller degree when expressed in terms of generators of $H_{k-1}$, so the induction hypothesis applies. 
\end{proof}

To conclude the proof of our theorem, we argue as follows. Each algebra $H_n$ is positively graded, so according to Proposition \ref{prop:free-alg-homol}, in order to establish its freeness it is enough to show that its cohomology with trivial coefficients vanishes in positive degrees. For that, we shall use the general result of Fresse \cite[Th.~17.3.4]{MR2494775} that gives a criterion for the isomorphism
 \[
H^{\calO}_\bullet(A,F)\cong{\mathrm{Tor}}_\bullet{U_{\calO}(A)}(F,\Omega^1_A)
 \]
to hold. 
(The question of identifying varieties of algebras for which the latter assertion is true is particularly meaningful in the context of Ginzburg's ``geometry over an operad'' \cite{MR1839485,Gin}.) This criterion requires that the operad $\calO$ is $\Sigma_*$-cofibrant (which is automatic over a field of zero characteristic) and that $\partial(\calO)$ and $\Omega^1_{\calO}$ are cofibrant as right $\calO$-modules, so Lemmas \ref{lm:partialfree} and \ref{lm:kahlerfree} show that it applies in our case (free right modules over an operad with zero differential are cofibrant), so we may compute the cohomology via the $\mathrm{Ext}$ functors over universal multiplicative enveloping algebras. 

Next, we note that, according to Lemma \ref{lm:freemodule}, for any $\calO$-algebra $A$ with zero operations and any subalgebra $H\subset A$, the universal enveloping algebra $U_\calO(A)$ is free as a $U_\calO(H)$-module. Because of the PBW property guaranteed by Lemma \ref{lm:partialfree}, the same is actually true for any algebra $A$ and its subalgebra $H$: we impose the usual filtration on $U_\calO(A)$ and take the associated graded algebra, then the freeness holds after taking the associated graded algebras, and we may lift the free generators to the original algebra. Thus, in particular, $U_\mathfrak{M}(F_n)$ is a free right $U_\mathfrak{M}(H_n)$-module. 

Let use base change theorem for Tor \cite[Th.~10.73]{MR2455920}: if $\phi\colon R\to S$ is a ring homomorphism, $M$ is a right $T$-module, and $B$ is a left $R$-module, then there is a spectral sequence 
 \[
E^2_{p,q}=\mathrm{Tor}^T_p(M,\mathrm{Tor}^R_q(T,B)) \Rightarrow \mathrm{Tor}_n^R(M,B). 
 \]
We shall apply this to $R=U_\mathfrak{M}(H_n)$, $S=U_\mathfrak{M}(F_n)$, $M=\k$ and $B=\Omega^1_{H_n}$. Since the trivial left $U_\mathfrak{M}(H_n)$-module is a restriction of the trivial left $U_\mathfrak{M}(F_n)$-module, this spectral sequence converges precisely to $H_n^\calO(H_n,\k)$. Let us note that since in our case $T$ is a free right $R$-module, we have 
 \[
\mathrm{Tor}^R_\bullet(T,B)=\mathrm{Tor}^R_0(T,B)\cong T\otimes_R B=U_\mathfrak{M}(F_n)\otimes_{U_\mathfrak{M}(H_n)}\Omega^1_{H_n} 
 \]
which is a left $U_\mathfrak{M}(F_n)$-submodule of the free left module $\Omega^1_{F_n}$. Submodules of free one-sided modules over associative algebras are free \cite{MR800091}, so 
 \[
\mathrm{Tor}^T_\bullet(M,T\otimes_R B)=\mathrm{Tor}^T_0(M,T\otimes_R B),
 \]
which implies that our spectral sequence collapses on the second page and the homology $H_\bullet^{\calO}(H_n,\k)$ vanishes in positive degrees, and therefore $H_n$ is free. According to Lemma \ref{lm:shirshov-trick}, the variety $\mathfrak{M}$ has the Nielsen--Schreier property.
\end{proof}

Let us record an immediate corollary of our result, which shows that an operad defining a Nielsen--Schreier variety of algebras cannot be ``small''. 

\begin{corollary}
Suppose that the variety of algebras encoded by the given operad $\calO$ is Nielsen--Schreier. If $\calX$ is the linear species of generators of the operad $\calO$, then for the exponential generating series of dimensions of species we have a coefficient-wise inequality 
 \[
f_{\calO}(t)\ge \int_0^t \frac{dt}{1-f'_{\calX}(t)}.
 \]
In particular, if $\calO$ is generated by $k$ linearly independent generators of arity $2$, we have $\dim\calO(n)\ge k^{n-1}(n-1)!$.   
\end{corollary}

\begin{proof}
It follows from Lemma \ref{lm:freepartial} that $\partial(\calO)$ is a free twisted associative algebra, and that $\partial(\calX)$ generates a free twisted associative subalgebra. It is well known \cite{MR1629341} that on the level of exponential generating functions, $\partial$ corresponds to derivative, and forming the free twisted associative subalgebra corresponds to composition with $\frac{1}{1-t}$, therefore we have a coefficient-wise inequality 
 \[
f'_{\calO}(t)\ge \frac{1}{1-f'_{\calX}(t)},
 \]
which is equivalent to the inequality we aim to prove. Additionally, if $\calO$ is generated by $k$ linearly independent generators of arity $2$, we have $f_{\calX}(t)=k\frac{t^2}{2}$, so $f'_{\calX}(t)=kt$, and 
 \[
\int_0^t\frac{dt}{1-f'_{\calX}(t)}=
\int_0^t\frac{dt}{1-kt}=-\frac{\log(1-kt)}{k},
 \]
and the claimed inequality follows.  
\end{proof}

We conjecture that the property (M1) alone is enough to ensure the Nielsen--Schreier property. Equivalently, we propose the following conjecture.

\begin{conjecture}\label{conj:SchrComb}
For a variety of algebras $\mathfrak{M}$ over a field of zero characteristic, the following properties are equivalent:
\begin{itemize}
\item $\mathfrak{M}$ has the Nielsen--Schreier property,
\item for every free $\mathfrak{M}$-algebra $A$, its universal multiplicative enveloping algebra $U_\mathfrak{M}(A)$ is a free associative algebra,
\item for the reverse graded path-lexicographic ordering, each leading term of the reduced Gr\"obner basis of the corresponding shuffle operad $\calO^f$ has the minimal leaf directly connected to the root.
\end{itemize}
\end{conjecture}

\section{Pre-Lie algebras}\label{sec:examplesfirst}

 Recall that the variety of pre-Lie algebras \cite{MR1827084}, also known as right-symmetric algebras, is defined by the identity
 $(a_1a_2)a_3-a_1(a_2a_3)=(a_1a_3)a_2-a_1(a_3a_2)$.
Existing results about pre-Lie algebras suggest that this variety might be Nielsen--Schreier. For instance, according to a result of Kozybaev, Makar-Limanov and the second author \cite{MR2431177}, two-generated subalgebras of free pre-Lie algebras are free. Moreover, in the context of our general result, it is worth recalling the result of Kozybaev and the second author \cite{MR2078718} (see also \cite{DFC,MR4381941}) that the underlying vector space of the universal multiplicative enveloping algebra of a pre-Lie algebra $L$ is isomorphic to $T(L)\otimes S(L)$, meaning that a PBW type theorem holds for universal multiplicative envelopes. We shall now show how to use the operad theory approach in this case.  

Let us remark that in \cite{KozMult} it is claimed that the variety of right-symmetric algebras does not have the Nielsen--Schreier property. Unfortunately, there is an issue with the two main proofs of that paper that rely on highly intricate computations. We studied the arguments of \cite{KozMult} in detail, and we believe that we identified the problematic parts. First, the claimed polynomial relation between particular five elements of the free two-generated algebra does not hold (we checked this using the \texttt{albert} software for computations in nonassociative algebras \cite{albert}; an independent verification was performed by Frederic Chapoton using \texttt{SageMath} \cite{sagemath}). Second, the proof of non-freeness of the multiplicative universal envelope of the free one-generated right-symmetric algebra seems to start with a correct identity but then makes a claim on algebraic independence that is false. 

\begin{theorem}\label{th:preLie}
The variety of pre-Lie algebras has the Nielsen--Schreier property.
\end{theorem}

\begin{proof}
Let us consider the operations $\alpha(a_1,a_2)=a_1a_2$ and $\beta(a_1,a_2)=a_2a_1$ which generate the operad of right-symmetric algebras as a shuffle operad. In terms of these operations, the linear basis of the $S_3$-module of consequences of the right-symmetric identity is given by the elements 
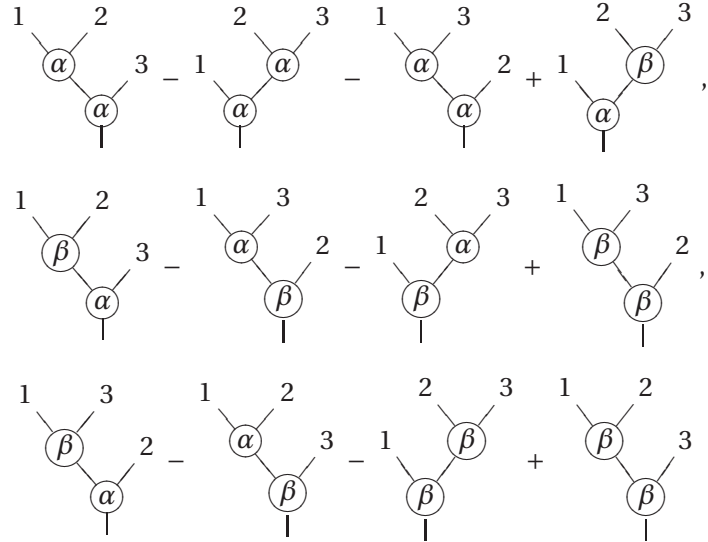
\begin{gather*}
\lbincomb{\alpha}{\alpha}{1}{2}{3}-\rbincomb{\alpha}{\alpha}{1}{2}{3}-\lbincomb{\alpha}{\alpha}{1}{3}{2}+\rbincomb{\alpha}{\beta}{1}{2}{3},\\
\lbincomb{\alpha}{\beta}{1}{2}{3}-\lbincomb{\beta}{\alpha}{1}{3}{2}-\rbincomb{\beta}{\alpha}{1}{2}{3}+\lbincomb{\beta}{\beta}{1}{3}{2},\\
\lbincomb{\alpha}{\beta}{1}{3}{2}-\lbincomb{\beta}{\alpha}{1}{2}{3}-\rbincomb{\beta}{\beta}{1}{2}{3}+\lbincomb{\beta}{\beta}{1}{2}{3}
\end{gather*}
corresponding to the more convenional expressions
\begin{gather*}
(a_1a_2)a_3-a_1(a_2a_3)=(a_1a_3)a_2-a_1(a_3a_2),\\
(a_2a_1)a_3-a_2(a_1a_3)=(a_2a_3)a_1-a_2(a_3a_1),\\
(a_3a_1)a_2-a_3(a_1a_2)=(a_3a_2)a_1-a_3(a_2a_1).
\end{gather*}
For the reverse graded path-lexicographic ordering corresponding to the ordering $\alpha>\beta$ of generators, the defining relations form a Gr\"obner basis with the leading terms
 $$
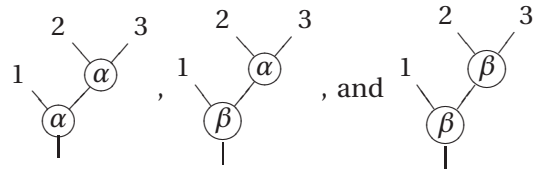
\rbincomb{\alpha}{\alpha}{1}{2}{3}, \rbincomb{\beta}{\alpha}{1}{2}{3}, \text{ and } \rbincomb{\beta}{\beta}{1}{2}{3}
 $$
satisfying the first combinatorial condition of Theorem \ref{th:SchrComb}. For the ordering obtained by superposition of the reversed lexicographic ordering on permutations of leaves and the graded lexicographic ordering on path sequences  corresponding to the ordering $\beta>\alpha$ of generators, the defining relations form a Gr\"obner basis with the leading terms
 $$
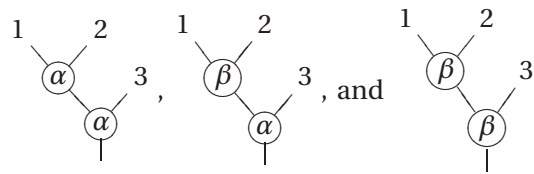
\lbincomb{\alpha}{\alpha}{1}{2}{3}, \lbincomb{\alpha}{\beta}{1}{2}{3}, \text{ and } \lbincomb{\beta}{\beta}{1}{2}{3}
 $$
satisfying the second combinatorial condition of Theorem \ref{th:SchrComb}. Both of these statements are easy to verify using the following observation outlined in \cite[Corollary 1]{MR3203367}. For a shuffle operad with quadratic relations $G$, the shuffle tree monomials for which each quadratic divisor is a leading term of an element of $G$ span the Koszul dual operad, and $G$ is a quadratic Gr\"obner basis if and only if the number of such shuffle tree monomials with $n$ leaves is equal to the dimension of the arity $n$ component of the Koszul dual operad. Indeed, in the first case, the corresponding shuffle tree monomials are right combs for which the path sequence of the maximal leaf is $\beta^i\alpha^{n-i-1}$ for some $0\le i\le n-1$, and in the second case the corresponding shuffle tree monomials are left combs with the leaves labelled $1$,\ldots, $n$ in the planar order, and the path sequence of the minimal leaf equal to $\alpha^i\beta^{n-i-1}$ for some $0\le i\le n-1$. Thus, in both cases, the cardinality of the spanning set for the arity $n$ component of the Koszul dual operad is equal to $n$. The Koszul dual operad of the operad of pre-Lie algebras is the operad $\mathop{\mathrm{Perm}}$ of permutative algebras, and it is well known that $\dim\mathop{\mathrm{Perm}}(n)=n$, so our relations form a Gr\"obner basis in both cases. Examining these Gr\"obner bases, we see that Theorem \ref{th:SchrComb} applies.
\end{proof}

In \cite{MR2431177}, it is shown that automorphisms of two-generated free pre-Lie algebras are tame. Theorem \ref{th:preLie}, when combined with Theorem \ref{th:tame}, immediately implies the following generalization. 

\begin{corollary}
Automorphisms of finitely generated free pre-Lie algebras are tame. 
\end{corollary}

\section{Algebras with two compatible Lie brackets}

Recall that an algebra with two compatible Lie brackets is a vector space $V$ equipped with two operations $a_1,a_2\mapsto[a_1,a_2]$ and $a_1,a_2\mapsto\{a_1,a_2\}$ which are skew-symmetric, satisfy the Jacobi identity individually, and additionally their sum also satisfies the Jacobi identity. The latter condition is equivalent to the identity
\begin{multline*}
[\{a_1, a_2\}, a_3]-[\{a_1, a_3\}, a_2]-[a_1, \{a_2,a_3\}]+\\+
\{[a_1, a_2], a_3\}-\{[a_1, a_3], a_2\}-\{a_1, [a_2,a_3]\}=0.
\end{multline*}
Since it is known \cite{MR2333979} that the dimension of the $n$-th component of the operad of two compatible Lie brackets is equal to $n^{n-1}$, which is also the dimension of the $n$-th component of the operad of pre-Lie algebras \cite{MR1827084}, the result of Theorem\ref{th:preLie} suggests that the variety of algebras with two compatible Lie brackets might have the Nielsen--Schreier property. We shall now show that it is indeed the case. 

\begin{theorem}\label{th:compatLie}
The variety of algebras with two compatible Lie brackets has the Nielsen--Schreier property.
\end{theorem}

\begin{proof}
For the reverse graded path-lexicographic ordering corresponding to the ordering \[[-,-]>\{-,-\}\] of generators, the defining relations form a Gr\"obner basis with the leading terms
 $$
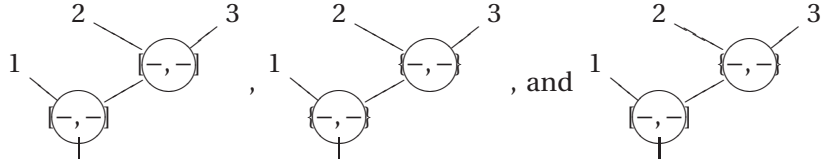
\rbincomb{[-,-]}{[-,-]}{1}{2}{3}, \rbincomb{\{-,-\}}{\{-,-\}}{1}{2}{3}, \text{ and } \rbincomb{[-,-]}{\{-,-\}}{1}{2}{3}
 $$
satisfying the first combinatorial condition of Theorem \ref{th:SchrComb}. 
For the graded path-lexicographic ordering corresponding to the ordering \[[-,-]>\{-,-\}\] of generators, the defining relations form a Gr\"obner basis with the leading terms
 $$
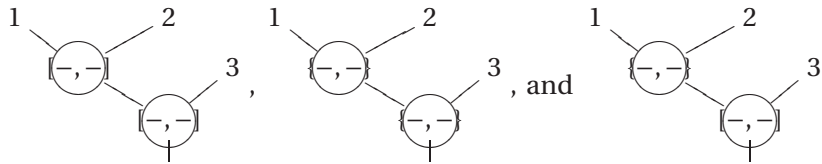
\lbincomb{[-,-]}{[-,-]}{1}{2}{3}, \lbincomb{\{-,-\}}{\{-,-\}}{1}{2}{3}, \text{ and } \lbincomb{[-,-]}{\{-,-\}}{1}{2}{3}
 $$
satisfying the second combinatorial condition of Theorem \ref{th:SchrComb}. (Both of these statements easily follow from the observation quoted in the proof of Theorem \ref{th:preLie}, together with the known fact that the arity $n$ component of the Koszul dual operad is equal to $n$.) Combining these observations completes the proof.
\end{proof}

An almost identical proof works for Lie algebras with several compatible Lie brackets \cite{MR2440264}.

\section{Varieties whose identities do not use substitutions of operations}

In this section we record a generalization of the results of Kurosh and Polin mentioned in the introduction. 

\begin{proposition}
Suppose that all identities of the variety $\mathfrak{M}$ are combinations of structure operations (no substitutions are used). Then the variety $\mathfrak{M}$ has the Nielsen--Schreier property.
\end{proposition}

\begin{proof}
In the language of operads, we are talking about free operads. Indeed, each structure operation with $k$ arguments \emph{a priori} generates the regular representation of the group $S_k$, and identities that are combinations of structure operations give a collection of submodules in the regular modules which have to be quotiented out. What remains is certain collection of representations of symmetric groups that generates our operad freely. In particular, Theorem \ref{th:SchrComb} applies tautologically, since there are no relations to consider. 
\end{proof}

It turns out that this proposition implies the Nielsen--Schreier property for the variety of Akivis algebras \cite{MR0405261,MR1726261}, first proved in \cite{MR1899864}. Recall that an Akivis algebra is an algebra with one skew-symmetric binary operation $[-,-]$ and one ternary operation $(-,-,-)$ satisfying the identitiy
 \[
[[a_1,a_2],a_3]+
[[a_2,a_3],a_1]+
[[a_3,a_1],a_2]=
\sum_{\sigma\in S_3} (-1)^{\sigma} (a_{\sigma(1)},a_{\sigma(2)},a_{\sigma(3)}).
 \]

\begin{corollary}
The variety of Akivis algebras has the Nielsen--Schreier property. 
\end{corollary}

\begin{proof}
The six-dimensional space of ternary generators of the corresponding operad is the regular representation of $S_3$ generated by $(-,-,-)$; as such, it splits into a direct sum of one copy of the trivial representation, one copy of the sign representation, and two copies of the two-dimensional irreducible representation. We note that the element 
 \[
\sum_{\sigma\in S_3} (-1)^{\sigma} (a_{\sigma(1)},a_{\sigma(2)},a_{\sigma(3)}).
 \]
found in the right hand side of the defining identity of the Akivis algebras is precisely the generator corresponding to the copy of the sign representation, and the Akivis identity allows one to eliminate this element, replacing it by the Jacobiator 
 \[
[[a_1,a_2],a_3]+
[[a_2,a_3],a_1]+
[[a_3,a_1],a_2].
 \]
This elimination implements an isomorphism between the operad of the Akivis algebras and the free operad generated by one skew-symmetric binary operation $[-,-]$ and a five-dimensional space of ternary operations where the $S_3$-action is the direct sum of the trivial representation and two copies of the two-dimensional irreducible representation, so it has the Nielsen--Schreier property.  
\end{proof}

A similar argument can prove that the variety of Sabinin algebras \cite{MR3174282,MR2304338,MR924255,MR1899864} has the Nielsen--Schreier property. Indeed, according to \cite[Th.~6.1]{MR3489597}, the operad of Sabinin algebras is free, so it has the Nielsen--Schreier property.

\section{Intersection of Nielsen--Schreier varieties}

Let us record a simple general observation that allows one to construct new Nielsen--Schreier varieties from known ones.

\begin{proposition}
Suppose that two varieties with disjoint sets of structure operations both satisfy the conditions of Theorem \ref{th:SchrComb}. Then the intersection of those varieties satisfy this condition as well. In particular, that intersection has the Nielsen--Schreier property. 
\end{proposition}

\begin{proof}
In the language of operads, the intersection of two varieties with disjoint sets of structure operations corresponds to the categorical coproduct, also known as the free product, of the corresponding operads. For the reverse path lexicographic ordering, the reduced Gr\"obner basis of such operad is the union of the two reduced Gr\"obner bases, assuring the condition (M1). Moreover, given two different orderings of shuffle trees decorated by two subsets of structure operations, one can find an ordering of shuffle trees decorated by their union that restricts to the two given orderings on the corresponding subsets of shuffle trees, implying that the reduced Gr\"obner basis of such operad is the union of the two reduced Gr\"obner bases and assuring the condition (M2). Thus, Theorem \ref{th:SchrComb} applies. 
\end{proof}

Our first observation is that this result applies to Lie-admissible algebras \cite{MR27750}. Recall that a Lie-admissible algebra is an algebra with one binary operation satisfying the identitiy
 \[
\sum_{\sigma\in S_3} (-1)^{\sigma} \left((a_{\sigma(1)}a_{\sigma(2)})a_{\sigma(3)}-a_{\sigma(1)}(a_{\sigma(2)}a_{\sigma(3)})\right)=0.
 \]

\begin{corollary}\label{cor:LieAdm}
The variety of Lie-admissible algebras has the Nielsen--Schreier property.
\end{corollary}

\begin{proof}
In terms of the operations $a_1\circ a_2=a_1a_2+a_2a_1$ and $[a_1,a_2]=a_1a_2-a_2a_1$, the Lie admissibility relation becomes the Jacobi identity for the second operation. It was observed in \cite{MR2225770} that this amounts to the fact that the operad of Lie-admissible algebras is the coproduct of the Lie operad and the free operad on one binary commutative operation. Thus, the variety of Lie-admissible algebras can be viewed as the intersection of the variety of Lie algebras and the variety of all commutative algebras, both of which satisfy the combinatorial conditions of Theorem~\ref{th:SchrComb}. 
\end{proof}

\section{Increasing degrees of identities}

Let us prove a generalization of a result of the second author who proved that the variety of algebras satisfying the identity $xx^2=0$ has the Nielsen--Schreier property.

\begin{theorem}\label{th:rightnil}
For every degree $n\ge 1$, the variety of algebras satisfying the right nil identity 
 \[
x(x(\cdots(xx^2)))=0
 \]
has the Nielsen--Schreier property.
\end{theorem}

\begin{proof}
This identity is equivalent to the multilinear one
 \[
\sum_{\sigma\in S_n} a_{\sigma(1)}(a_{\sigma(2)}(\cdots(a_{\sigma(n-2)}(a_{\sigma(n-1)}a_{\sigma(n)}))))=0.
 \]
Let us use, once again, the symmetrized and the skew-symmetrized operations $a_1\circ a_2=a_1a_2+a_2a_1$ and $[a_1,a_2]=a_1a_2-a_2a_1$. For the reverse graded path-lexicographic ordering corresponding to the ordering of generators prioritizing $[-,-]$, this identity has as the leading term the only right comb with all the vertices but the one at the top labelled $[-,-]$, and the top vertex labelled $-\circ-$. For the graded path-lexicographic ordering corresponding to the ordering  of generators prioritizing $[-,-]$, this identity has as the leading term the left comb with all the vertices but the one at the top labelled $[-,-]$, the top vertex labelled $-\circ-$, and the leaves labelled $1$,\ldots, $n$ in the planar order. In each case, there are no self-overlaps, so we obtain a Gr\"obner basis, and Theorem \ref{th:SchrComb} applies.
\end{proof}

\section{Deformation of the right nil identity}

Our methods also lead to the following rather striking generalization of Theorem \ref{th:rightnil} that exhibits multiparametrics families of Nielsen--Schreier varieties. For a planary binary tree $T$ with $n$ leaves, let $m_T(x_1,\ldots,x_n)$ denote the nonassociative product of $x_1,\ldots,x_n$ with the bracketing according to $T$. For example, for 
 \[
T=\binfork{}{}{}{}{}{}{},
 \]
we have $m_T(x,x,x,x)=((xx)(xx))=(x^2)^2$. To state the following theorem, we fix a positive integer $n>2$, and consider an arbitrary way to assign a scalar $b_T\in\k$ to every planar binary tree $T$ with $n$ leaves. 

\begin{theorem}\label{th:parametric}
For a Zariski open set in the affine space with coordinates $b_T$, the variety of algebras satisfying the identity 
 \[
\sum_T b_T m_T(x,\ldots,x)=0
 \] 
has the Nielsen--Schreier property.
\end{theorem}

\begin{proof}
This identity is equivalent to the multilinear identity
 \[
\sum_{\sigma\in S_n} b_T m_T(a_{\sigma(1)},\ldots,a_{\sigma(n)})=0.
 \]
As before, it will be convenient to present our variety via the symmetrized and the skew-symmetrized operations $a_1\circ a_2=a_1a_2+a_2a_1$ and $[a_1,a_2]=a_1a_2-a_2a_1$. 

For the reverse graded path-lexicographic ordering corresponding to the ordering of generators  prioritizing $[-,-]$, there exists a Zariski open set in the affine space with coordinates $b_T$ for which this identity has as the leading term the only right comb with all the vertices but the one at the top labelled $[-,-]$, and the top vertex labelled $-\circ-$. Indeed, the coefficient of that term is a certain linear combination of $b_T$, so our Zariski open set is a complement of a hyperplane. Similarly, for the graded path-lexicographic ordering corresponding to the ordering of generators  prioritizing $[-,-]$, there exists a Zariski open set in the affine space with coordinates $b_T$ for which this identity has as the leading term the left comb with all the vertices but the one at the top labelled $[-,-]$, the top vertex labelled $-\circ-$, and the leaves labelled $1$,\ldots, $n$ in the planar order. In each case, there are no self-overlaps, so we obtain a Gr\"obner basis, and Theorem \ref{th:SchrComb} applies.
\end{proof}

Let us examine the cases $n=3$ and $n=4$ in some detail. 

For $n=3$, we are dealing with the identity $\alpha x^2x+\beta xx^2=0$. A simple computation shows that for $\alpha=\beta$ this variety is closely related to the variety of mock Lie algebras and is not Nielsen--Schreier, while for $\alpha\ne\beta$ this variety is Nielsen--Schreier. However, this example is producing a parametric family that is only superficially infinite: in fact, for $\alpha\ne\pm\beta$ all the corresponding operads are pairwise isomorphic, and hence isomorphic to the second author's variety $xx^2=0$, via the automorphism of the free operad on one binary generator sending the generator $a_1a_2$ to $a_1a_2+\lambda a_2a_1$; in the language of varieties this corresponds to the notion of quasi-equivalence going back to the work of Albert\cite{MR32609}. Thus, the only new example obtained this way is the variety of algebras satisfying the identity $xx^2=x^2x$. (Note that the Nielsen--Schreier property obviously fails for the variety of \emph{all} power associative algebras, which is defined, by a remarkable result of Albert \cite{MR26044,MR27750} by the above identity together with just one extra identity $(x^2x)x=x^2x^2$.) 

For $n=4$, we are dealing with the identity 
 \[
\alpha_1 (x^2x)x+\alpha_2(xx^2)x+\alpha_3(x^2)^2+\alpha_4 x(x^2x)+\alpha_5x(xx^2)=0. 
 \]
Up to re-scaling, there are four algebraically independent parameters, which, since the automorphisms of the free operad depend on only one parameter, immediately implies the following unexpected result.

\begin{corollary}
There are infinitely many pairwise not quasi-equivalent Nielsen--Schreier varieties of algebras defined by identities of degree four.
\end{corollary}

\section{Alia and one-sided alia algebras}\label{sec:exampleslast}

Recall that alia (anti-Lie-admissible) algebras \cite{MR2676255} are the algebras with the following identity for the symmetrized and the skew-symmetrized operations:
 \[
[a_1,a_2]\circ a_3+[a_2,a_3]\circ a_1+[a_3,a_1]\circ a_2=0.
 \]
In the same paper one finds the definition of a left alia algebra as the algebra satisfying the identity
 \[
[a_1,a_2] a_3+[a_2,a_3] a_1+[a_3,a_1] a_2=0,
 \]
and the ``opposite'' definition of a right alia algebra as the algebra satisfying the identity
 \[
a_3[a_1,a_2]+a_1[a_2,a_3]+a_2[a_3,a_1]=0.
 \]
We shall show that the corresponding varieties have the Nielsen--Schreier property, proving the following general result.

\begin{theorem}
For any $\alpha$, the variety of algebras satisfying the identity
 \[
a_3[a_1,a_2]+a_1[a_2,a_3]+a_2[a_3,a_1]+\alpha([a_1,a_2] a_3+[a_2,a_3] a_1+[a_3,a_1] a_2)=0
 \]
has the Nielsen--Schreier property.
\end{theorem}

\begin{proof}
For $\alpha=-1$, we obtain the variety of Lie-admissible algebras, so Corollary \ref{cor:LieAdm} applies. Suppose that $\alpha\ne-1$. Analogous to the proof of Theorem \ref{th:parametric}: if we write everything in terms of the symmetrized and the skew-symmetrized operations, we obtain, for the two orderings of interest, the leading terms
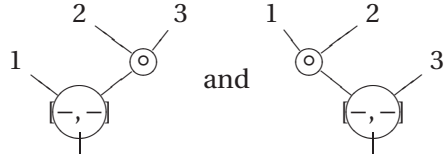
\[
\rbincomb{[-,-]}{\circ}{1}{2}{3} \text{ and } \lbincomb{[-,-]}{\circ}{1}{2}{3}
 \] 
respectively. In each case, there are no self-overlaps, so we obtain a Gr\"obner basis, and Theorem \ref{th:SchrComb} applies. (In fact, for $\alpha\ne -1,1$ the corresponding operads are all pairwise isomorphic.)
\end{proof}

\section*{Funding} This work was supported by Institut Universitaire de France, by Fellowship of the University of Strasbourg Institute for Advanced Study through the French national program ``Investment for the future'' [IdEx-Unistra, grant USIAS-2021-061 to V.D.]; the French national research agency [grant ANR-20-CE40-0016 to V.D]; and the Ministry of Education and Science of the Republic of Kazakhstan [grant number AP14872073 to U.U.].

\section*{Acknowledgements } We are grateful to Leonid Bokut for insight into the history of the Dniester Notebook and to Frederic Chapoton for an independent check of the erroneous claim we discovered in \cite{KozMult}. We also thank Xabier Garc\'ia-Mart\'inez, Victor Ginzburg, Anton Khoroshkin, Ivan Shestakov, Bruno Vallette, and Pedro Tamaroff for discussions of a draft version of this paper. Finally, it is a pleasure to thank the participants of the seminar of the team ART (Alg\`ebre, Repr\'esentations, Topologie) and online seminar \texttt{LieJor}, in particular Vesselin Drensky, Christian Kassel, and Victor Petrogradsky, for interesting questions and remarks on the talks about our results. 

A significant part of this work was completed during the visits of the Max Planck Institute for Mathematics in Bonn (of both authors) and the Centre de Recerca Matem\`atica in Barcelona (of the first author); we wish to express our gratitude to those institutions for hospitality and excellent working conditions.

\bibliographystyle{plain}
\bibliography{biblio}

\end{document}